\documentclass[12pt, a4paper]{amsart}

\usepackage{amsthm}
\usepackage{amsmath}
\usepackage{mathtools}
\usepackage{amssymb}
\usepackage{latexsym}
\usepackage{graphics}
\usepackage{graphicx}
\usepackage{epstopdf}
\usepackage{tikz}
\usepackage{color}
\usepackage{import}

\usepackage{epstopdf}
\usepackage{float}
\usepackage[normalem]{ulem}

\usepackage[hidelinks,pdftex]{hyperref}

\AtBeginDocument{%
   \def\MR#1{}
}

\usepackage{marginnote}

\long\def\@savemarbox#1#2{\global\setbox#1\vtop{\hsize\marginparwidth
  \@parboxrestore\tiny\raggedright #2}}
\numberwithin{equation}{section}
\theoremstyle{plain}
\newtheorem{theorem}[equation]{Theorem}

\newtheorem{lemma}[equation]{Lemma}

\newtheorem{proposition}[equation]{Proposition}

\newtheorem*{namedtheorem}{\theoremname}
\newcommand{\theoremname}{testing}

\theoremstyle{definition}
\newtheorem{definition}[equation]{Definition}
\newtheorem{remark}[equation]{Remark}
\newtheorem{example}[equation]{Example}
\newtheorem{assumption}[equation]{Assumption}

\newcommand{\bdy}{\partial}

\newcommand{\refthm}[1]{Theorem~\ref{Thm:#1}}
\newcommand{\reflem}[1]{Lemma~\ref{Lem:#1}}
\newcommand{\refprop}[1]{Proposition~\ref{Prop:#1}}

\newcommand{\refrem}[1]{Remark~\ref{Rem:#1}}

\newcommand{\refeqn}[1]{\eqref{Eqn:#1}}
\newcommand{\refitm}[1]{\eqref{Itm:#1}}
\newcommand{\refdef}[1]{Definition~\ref{Def:#1}}
\newcommand{\refsec}[1]{Section~\ref{Sec:#1}}
\newcommand{\reffig}[1]{Figure~\ref{Fig:#1}}

\title[Standard position]{Standard position for surfaces in link complements in arbitrary 3-manifolds}
\author{Jessica S. Purcell}
\author{Anastasiia Tsvietkova}
\date{}
\subjclass[2010]{}

\thanks{23 August 2023}

\everymath{\displaystyle}

\begin{document}

\begin{abstract}
Since the 1980s, it has been known that essential surfaces in alternating link complements can be isotoped to be transverse to the link diagram almost everywhere, with the exception of some well-understood intersections, and described combinatorially as a result.
This was called standard position for surfaces and has had numerous applications. However, the original techniques only apply to classical alternating links projected onto the 2-sphere inside the 3-sphere. 
In this paper, we prove that standard position for surfaces can be extended to a broader class, namely weakly generalized alternating links. Such links include all classical prime non-split alternating links in the 3-sphere, and also many links that are alternating on higher genus surfaces, or lie in manifolds besides the 3-sphere. As an application, we show that all such links are prime, and that under mild restrictions, essential Conway spheres for such links interact with the diagram exactly as in the classical alternating setting.
\end{abstract}

\maketitle

\section{Introduction}\label{Sec:Intro}

Essential embedded surfaces in a 3-manifold can provide useful information about the topology of that manifold. For example, Haken manifolds, which contain an orientable closed essential surface, have been an object of significant research since they were introduced in the 1960's and have many important properties \cite{Haken, Waldhausen}. To work with essential embedded surfaces, it is often useful to put them in a certain topological form that relates to properties of the ambient 3-manifold. When the ambient 3-manifold is an alternating link complement, Menasco and Thistlethwaite showed that an essential surface could be isotoped into \emph{standard position} with respect to the link projection surface~\cite{Menasco1984, MenascoThistlethwaite}. For more general 3-manifolds, an essential surface may be isotoped into \emph{normal form} with respect to a triangulation of a 3-manifold~\cite{Haken}, or similarly normal form with respect to polyhedral or other decompositions~\cite{lac00, fg09}. In \cite{HowiePurcell}, Howie and Purcell extended normal form to broad generalizations of alternating links, called weakly generalized alternating links, and used this form to determine geometric information, generalizing several of the above results.

In this paper, we take this generalization further. Menasco in~\cite{Menasco1984} and Menasco and Thistlethwaite in~\cite{MenascoThistlethwaite} proved that essential surfaces in alternating link complements could be isotoped so that they closely interact with the combinatorics of the link diagram. Here, we generalize these results to weakly generalized alternating links.

As a corollary, we show that all weakly generalized alternating links are prime, extending a result of Howie and Purcell~\cite[Corollary~4.7]{HowiePurcell}, who showed primeness of such links with some restrictions.

We also classify essential Conway spheres for weakly generalized alternating links satisfying mild restrictions, showing that they either intersect the diagram in a single curve meeting four punctures (a ``visible'' sphere), or intersect it in two well-behaved curves (a ``hidden'' sphere). The visible--hidden dichotomy is analogous to a similar result for classical alternating links, discovered by Menasco~\cite{Menasco1984}, named and refined by Thistlethwaite~\cite{Thistlethwaite}, and used to study alternating tangles by others, for example in other collaborations by the authors: Hass, Thompson, and Tsvietkova~\cite{HTT3}, and Champanerkar, Kofman, and Purcell~\cite{CKP:RAK}.

\subsection{Alternating links in the 3-sphere and beyond}
Normal and standard form for surfaces in the complement of alternating links in the 3-sphere have had many applications. Historically, Menasco used it to prove a number of results for alternating link complements in the 3-sphere, such as classifying when the link is prime, split and hyperbolic \cite{Menasco1984}, and determining when a surface is incompressible \cite{Menasco1985}. Menasco and Thistlethwaite used it to show that the only reducible Dehn fillings are the expected ones on $(2,q)$-torus knots~\cite{MenascoThistlethwaite}, and thus to prove the cabling conjecture for all alternating knots. Lackenby used normal form to describe exceptional Dehn fillings \cite{lac00}. Howie used it to show that essential embedded surfaces with boundary can be used to detect classical alternating links \cite{Howie2017}. Hass, Thompson and Tsvietkova used it to give universal polynomial bounds for the number of embedded surfaces in alternating link complements \cite{HTT1, HTT2}. These are just some examples of numerous applications to alternating links in the 3-sphere.

In recent years, there has been significant interest in the study of alternating projections onto surfaces besides the 2-sphere, for example due to Adams~\cite{Adams:ToroidAlt}, Hayashi~\cite{Hayashi}, Ozawa~\cite{oza06}, Howie~\cite{how15t}, and others. There has also been interest in alternating knots within manifolds besides the 3-sphere, such as virtual knots, for example by Adams~\cite{Adams:ThickenedSfces}, Champanerkar, Kofman, and Purcell~\cite{CKP:Biperiodic}, and Howie and Purcell~\cite{HowiePurcell}.
The weakly generalized alternating links of this paper match those of Howie and Purcell. They include classical alternating knots on a 2-sphere in the 3-sphere, and also include many of the other generalizations above, including families of toroidally alternating knots and alternating knots on other Heegaard surfaces, alternating knots in thickened surfaces, i.e.\ virtual alternating knots, and further examples still. The full definition of such links is given in \refsec{WGA}.

\subsection{Meridianal surfaces in meridianal form}
One of the key results in this paper is to show that essential surfaces meeting a knot in meridians can be simultaneously isotoped into normal form without distorting the boundary curves. We call this \emph{meridianal form}, defined in \refsec{SBP}.

Menasco showed that any closed essential surface embedded in a non-split prime alternating link complement in $S^3$ contains a closed curve isotopic to the link meridian~\cite{Menasco1984}. This sometimes is referred to as the Meridian Lemma. By compressing along a meridianal annulus, it allows one to reduce the study of closed essential surfaces to the study of essential surfaces with meridianal boundary within the complement of a regular neighborhood of the link, i.e.\ the link exterior. 
A meridian lemma also holds for certain weakly generalized alternating links due to Howie and Purcell~\cite[Lemma~4.9]{HowiePurcell}, and has been announced for weakly generalized alternating links that are also virtual alternating links by Wei Lin~\cite{WeiLin}. A meridian lemma holds for other knots, such as algebraic knots~\cite{Ozawa:Algebraic}. For knots that satisfy such a lemma, a meridian compression yields a surface that can be put into meridianal form.
This was used, for example, in the proof of a polynomial upper bound on the number of closed surfaces in alternating links by Hass, Thompson and Tsvietkova~\cite{HTT1}, and in Lozano and Przytycki's work on 3-braid links~\cite{LP}.

Surfaces that meet a knot exterior in meridians of the knot also include meridianal annuli. These are essential annuli with both boundary components forming meridians of the knot.

By definition, if a link exterior admits an essential meridianal annulus, the link is not prime.
Studying such surfaces in the past led to results that a classical alternating link is prime if and only if its diagram is prime, due to Menasco~\cite{Menasco1984}.
In this paper, we show that all weakly generalized alternating links must be prime in \refthm{Prime}. This is analogous to Menasco's result, since weakly generalized alternating links have a diagrammatic primeness condition built into the definition.
Howie and Purcell proved a similar result for weakly generalized alternating links with an additional condition (namely $\hat{r}>4$; see \refsec{WGA})~\cite[Corollary~4.7]{HowiePurcell}.
Howie and Purcell's result already implies that all virtual alternating links that are checkerboard colorable and have a weakly prime diagram must also be prime. For example this is true for virtual alternating links whose diagrams have all complementary regions disks, as studied by Adams \emph{et~al}~\cite{Adams:ThickenedSfces}. However, the results in this paper extend primeness to broader families of weakly generalized alternating links, for example on projection surfaces that are compressible, which do not satisfy Howie and Purcell's stronger requirement.

Note that in the virtual setting, other notions of primeness also appear in the literature, for example recent work of Kindred~\cite{Kindred:Prime}.

Finally, surfaces meeting a knot exterior in meridians include 4-punctured spheres, which separate a knot into tangles. Examining such surfaces has allowed the detection and study of prime tangles, for example by Menasco and Thistlethwaite~\cite{Menasco1984, Thistlethwaite}, and more recently by Hass, Thompson and Tsvietkova~\cite{HTT3}, and Champanerkar, Kofman, and Purcell~\cite{CKP:RAK}.
The tools of this paper also allow us to extend such results to larger families of links. We consider essential 4-punctured spheres in the complement of weakly generalized alternating links satisfying an extra condition (namely $\hat{r}(\pi(L),\Pi)>4$), and we show that they satisfy the same constraints as in the classical alternating case; see \refthm{EssentialTangles}. In particular, they result in either a ``visible'' tangle (one $PPPP$ curve in standard position in Menasco's language), or a ``hidden'' tangle (two $PSPS$ curves); see \refsec{ConwaySpheres} for definitions.
We expect this to lead to further study of tangles for knots projected onto surfaces beyond the 2-sphere.

\subsection{Cyclic words associated to a surface}

We extend standard position to closed surfaces and meridianal surfaces, but also to surfaces with non-meridianal boundary.

One of the benefits of standard position for classical alternating links is that it allows us to translate topological properties of surfaces into combinatorial properties of their curves of intersection with the (slightly modified) projection surface for a link. Such intersections were encoded in \cite{Menasco1984} by letters $P$ (for link intesections) and $S$ (for crossing ball intersections). Every curve of intersection was put in correspondence with a cyclic word in such letters. Then possible words were investigated. In \cite{HTT1}, this was extended further to surfaces with boundary, now with letters $B$ (standing for arriving at the boundary of the surface) and $S$. We generalize this to weakly generalized alternating links in this paper.

This combinatorial approach played an important role in several results for alternating links in the 3-sphere, including many noted above. In addition, it has been used to classify genus two surfaces \cite{Thistlethwaite} and to find incompressible surfaces in alternating link complements in \cite{Menasco1985}.

In a subsequent paper, we apply this work to the problem of bounding the total number of embedded essential surface in these link complements~\cite{PT}. We apply these results together with combinatorial and geometric arguments to give a polynomial upper bound on the number of embedded surfaces in a wide class of cusped 3-manifolds, namely weakly generalized alternating link complements and their Dehn fillings.
Our upper bound only depends on the crossing number of the diagram, and is independent of the 3-manifold otherwise. These are the first bounds independent of the 3-manifold besides those for classical alternating links in $S^3$ in \cite{HTT1, HTT2};
in other existing work, the bound depends on the 3-manifold.

\subsection{Organization}

In \refsec{WGA} we recall the definition of weakly generalized alternating links. Their primary feature is recalled in \refsec{Chunks}, namely, they have a decomposition into checkerboard \emph{chunks} that extend much of the useful decomposition of alternating links into topological polyhedra via checkerboard decomposition, which Thurston linkened to gears of a machine~\cite{Thurston:notes}. We recall the definition of normal surfaces in chunks in \refsec{NormalSfc}. 
In \refsec{SBP} we describe the labelling of intersections of such surfaces with the chunk boundary by letters $P$, $B$, and $S$.

Our first main result, on meridianal form for normal surfaces, is proved in \refsec{Meridianal}. In \refsec{Standard}, we compile a few results on surfaces in standard form, including a result that shows that such a surface cannot meet the same saddle twice in a simple way.

In \refsec{Area}, we recall the notion of combinatorial area, introduced in this setting by Howie and Purcell~\cite{HowiePurcell}. Combining combinatorial area results with the labelling of boundary curves by letters $P$, $B$, and $S$, we find that the only zero area surfaces in a weakly generalized alternating link complement (with some restrictions) are disks with certain words on their boundaries. Sections~\ref{Sec:SSSS}, \ref{Sec:PPPP}, and~\ref{Sec:BBBB} are devoted to ruling out instances of zero area disks, or controlling when they occur. 
As a corollary of the work of \refsec{PPPP}, we prove the result on prime links.

Finally, in \refsec{ConwaySpheres}, we apply our results to essential Conway spheres to reproduce a result of Menasco~\cite{Menasco1984}, that such spheres are either ``visible'' or ``hidden'', with terminology due to Thistlethwaite~\cite{Thistlethwaite}.

\subsection{Acknowledgments}
We thank an anonymous referee whose suggestions have greatly improved the paper.
Purcell was partially supported by the Australian Research Council, grants DP160103085 and DP210103136. Tsvietkova was partially supported by the National Science Foundation (NSF) of the United States, grants DMS-2142487 (CAREER), DMS-1664425 (previously 1406588) and DMS-2005496, the Institute of Advanced Study under NSF grant DMS-1926686, and the Okinawa Institute of Science and Technology.

\section{Weakly generalized alternating link complements}\label{Sec:WGA}

In this section, we follow Howie and Purcell~\cite{HowiePurcell} to introduce a broad class of links that each have a diagram that is alternating on a closed surface $\Pi$ embedded in a 3-manifold $Y$, possibly with boundary.
As noted in the introduction, this is a very broad class that includes classical alternating knots, but also alternating knots on a Heegaard torus or more general Heegaard surface, originally studied by Adams~\cite{Adams:ToroidAlt} and Hayashi~\cite{Hayashi}, virtual alternating knots~\cite{Adams:ThickenedSfces, CKP:Biperiodic}, and knots with alternating diagrams on any embedded surface in any compact orientable 3-manifold.

\begin{assumption}\label{Assumption:3mfld}
    Throughout this paper, we will always require the 3-manifold $Y$ to be compact, orientable, and irreducible. The projection surface $\Pi$ is required to be a closed, orientable surface. If $Y$ has boundary, we will require $\bdy Y$ to be incompressible in $Y-N(\Pi)$, where $N(\cdot)$ always denotes a regular neighborhood. We further require $Y-N(\Pi)$ to be irreducible.
\end{assumption}

Given a link $L$ in $Y$, the \emph{link complement} is the manifold $Y-L$. The \emph{link exterior} is the compact manifold $Y-N(L)$. Note that if $Y$ is closed, then the link complement is homeomorphic to the interior of the link exterior.

\subsection{Generalized projection}

A \emph{generalized projection surface} $\Pi$ is a (possibly disconnected) oriented surface embedded in $Y$ so that $Y-\Pi$ is irreducible.
The connected components of $\Pi$, denoted $\Pi_1, \dots, \Pi_p$ are closed two-sided (orientable) surfaces.  
Let $N(\Pi) = \Pi\times (-1,1)$ denote a regular neighborhood.
For each component $\Pi_i$ of $\Pi$, define $\Pi_i^{\pm}$ to be $\Pi_i\times\{\pm 1\} \subset N(\Pi)$. Denote $\bigcup \Pi_i^+$ by $\Pi^+$ and similarly for $\Pi^-$. 

Since $Y-\Pi$ is irreducible, if some $\Pi_i$ is a 2-sphere, then $\Pi$ is homeomorphic to $S^2$, and $Y$ is homeomorphic to $S^3$. Let $L$ be a link that can be projected onto $\Pi$ in general position. That is, $L$ can be isotoped through $Y$ to lie in $N(\Pi)$ so that the image of the projection $\pi(L)$ consists of crossings and arcs between them on the surface $\Pi$. 
We call $\pi(L)$ a \emph{generalized diagram}, or simply a \emph{diagram}. 

Whenever in the paper we mention $\Pi$ or $\pi(L)$, we mean a generalized projection surface and the generalized diagram of $L$ on $\Pi$. We will also use the following terms.
\begin{itemize}
\item An arc of $\pi(L)$ (on $\Pi$) between two crossings is called an \emph{edge} of the diagram.
\item A \emph{crossing arc} is a simple arc in the complement of $L$ running from an overpass to an underpass of a crossing.
\item A \emph{region} of the diagram is a complementary region of the projection of $\pi(L)$ to $\Pi$. It is bounded by edges of the diagram.
\end{itemize}

Every knot has a very simple generalized diagram on the torus boundary of a regular neighborhood of the knot. To ensure our diagrams are sufficiently complicated, in a way that depends on $Y$ and $\Pi$, we introduce the notion of representativity. 
Define $r^{\pm}(\pi(L), \Pi_i)$ to be the minimum number of intersections between the projection of $\pi(L)$ onto $\Pi_i^{\pm}$ and the boundary of any essential compressing disk for $\Pi_i^{\pm}$ in $Y- \Pi$. If there are no essential compressing disks for $\Pi_i^{\pm}$ in $Y- \Pi$, then set $r^{\pm}(\pi(L), \Pi_i)=\infty$. The \emph{representativity} $r(\pi(L), \Pi)$ is the minimum of all values of $r^-(\pi(L), \Pi_i)$ and $r^+(\pi(L), \Pi_i)$, over all $i$. By definition, a sphere $S^2$ embedded in $S^3$ admits no essential compressing disk, so a usual alternating diagram has representativity $\infty$.

\begin{example}
\reffig{CheckColorable}, left, which is modified from a figure in \cite{HowiePurcell}, shows a diagram $\pi(L)$ on a torus $\Pi$. The representativity of the diagram depends on the manifold $Y$ and the embedding of $\Pi$ into $Y$. For example, first let $Y=S^3$, and embed the torus $\Pi$ as the standard Heegaard torus for $S^3$, with the vertical red curve shown mapping to a meridian of one solid torus in the Heegaard splitting of $S^3$ and the horizontal red curve mapping to a meridian of the other solid torus. Then $r^+(\pi(L),\Pi)=3$ and $r^-(\pi(L),\Pi)=0$, so $r(\pi(L),\Pi)=0$.
\begin{figure}[h]
  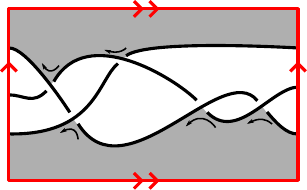
  \hspace{.2in}
  \includegraphics{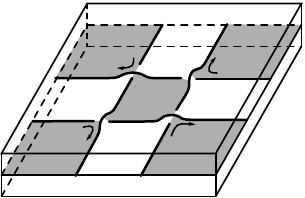}
\caption{Left: An example of an alternating diagram on a torus. The representativity will depend on the embedding of the torus into $Y$. In any case, the diagram is not checkerboard colorable. Right: A checkerboard colorable diagram.}
  \label{Fig:CheckColorable}
\end{figure}

If instead we let $Y=T^2\times [-1,1]$, the thickened torus, and we embed the torus $\Pi$ as the surface $T^2\times\{0\}$, then $\Pi$ admits no essential compressing disk, so $r^+(\pi(L),\Pi) = r^-(\pi(L),\Pi)=\infty$.
\end{example}

Finally, define the hat-representativity, $\hat{r}(\pi(L), \Pi)$, to be the minimum of
\[ \bigcup_i \max\{ r^-(\pi(L), \Pi_i), r^+(\pi(L), \Pi_i) \}. \]
Thus for example a surface with no compressing disks on one side will have infinite hat-representativity.

A generalized diagram $\pi(L)$ is said to be \emph{alternating} if for each region of $\Pi-\pi(L)$, each boundary component of the region is alternating, i.e.\ it can be given an orientation such that crossings run from under to over in the direction of orientation. An alternating generalized diagram $\pi(L)$ is said to be \emph{checkerboard colorable} if each region of $\Pi-\pi(L)$ can be oriented so that the induced orientation on each region's boundary is alternating: crossings run from under to over in the direction of orientation. Given a checkerboard colorable diagram, regions on opposite sides of an edge of $\pi(L)$ will have opposite orientations. We can color all regions with one orientation white, and all regions with the opposite orientation shaded.

\begin{example}
\reffig{CheckColorable}, left, shows a diagram that is alternating on a torus. For the outer annular region $A$ of the diagram, the figure shows an orientation assigned to its boundary. The crossings run from under to over with this orientation. Hence the region is alternating. However, this orientation is not induced by an orientation on $A$. If we choose an orientation on $A$, and then take the induced orientation on $\bdy A$, one boundary component will be oriented so that crossings run under to over, and the other will be oriented so that crossings run over to under. Hence this example is not checkerboard colorable. Note it will not be checkerboard colorable regardless of the manifold $Y$ it is embedded within.

However, \reffig{CheckColorable}, right, shows another diagram of an alternating link on a torus, embedded in $Y=T^2\times I$. This link diagram is checkerboard colorable.
\end{example}

Many results in the literature on alternating links require a reduced or simplified diagram; for example \cite{Menasco1984, HTT1, HTT2}. For a classical alternating link on $S^2\subset S^3$, the condition we need is diagrammatic primeness: if an essential curve intersects the diagram exactly twice, then it bounds a region of the diagram containing a single embedded arc. Diagrammatic primeness rules out connected sums of knots and nugatory crossings.
Analogously, a generalized diagram $\pi(L)$ on generalized projection surface $\Pi = \bigcup \Pi_i$ is \emph{weakly prime} if whenever $D\subset \Pi_i$ is a disk with $\bdy D$ intersecting $\pi(L)$ transversely exactly twice, either the disk $D$ contains a single embedded arc, or $\Pi_i$ is a 2-sphere and there is a single embedded arc on $\Pi_i-D$.

\subsection{Weakly generalized alternating link}\label{Sec:WGAKnots}

The diagram $\pi(L)$ on $\Pi$ is said to be a \emph{weakly generalized alternating link diagram} if
\begin{enumerate}
\item\label{Itm:Alternating} $\pi(L)$ is alternating on $\Pi$,
\item $\pi(L)$ is weakly prime,
\item\label{Itm:Connected} $\pi(L)\cap \Pi_i \neq \emptyset$ for each component $\Pi_i$ of $\Pi$.
\item\label{Itm:slope} each component of $L$ projects to at least one crossing in $\pi(L)$.
\item\label{Itm:CheckCol} $\pi(L)$ is checkerboard colorable, and
\item\label{Itm:Rep} the representativity $r(\pi(L), \Pi)\geq 4$.
\end{enumerate}
We say that a link $L$ is \emph{weakly generalized alternating} if it has a weakly generalized alternating diagram. 

Note these conditions were originally enumerated by Howie~\cite{how15t}, and in some sense are as general as possible to obtain prime alternating diagrams on embedded surfaces in $S^3$ in Howie's setting. Perhaps the most mysterious condition is item~\refitm{Rep}, the representativity condition. The restriction on representativity means that small surfaces such as compressing disks and meridianal annuli must be split by the projection surface $\Pi$ into disks parallel to $\Pi$, and thus they can be isotoped to lie entirely in $N(\Pi)$. The representativity condition is used throughout Howie--Purcell~\cite{HowiePurcell}, and consequently used throughout this paper. Recall as well that classical alternating knots and virtual alternating knots have infinite representativity; the representativity condition only needs to be checked when $\Pi$ is compressible, for example if it is a Heegaard surface.

From now on, every link we consider will be weakly generalized alternating.
Note that a classical reduced, prime, alternating diagram of a link $L$ on $\Pi=S^2$ in $S^3$ is an example of a weakly generalized alternating link diagram. The example of \reffig{CheckColorable}, right, is also a weakly generalized alternating link diagram on $\Pi = T^2\times\{0\}$ in $Y=T^2\times[-1,1]$.

These conditions are enough to guarantee that the link exteriors are irreducible and boundary irreducible~\cite[Corollary~3.16]{HowiePurcell}, which we will use below.

\section{Decomposition of a link complement into chunks}\label{Sec:Chunks}

Knot and link complements alternating on the projection sphere in $S^3$ have a well-known decomposition into topological polyhedra, suggested by W.~Thurston and described by Menasco \cite{men83}; see also \cite{lac00} or \cite[Section~11.1.1]{Purcell:HKT}. A more general decomposition into angled blocks was defined by Futer and Gu{\'e}ritaud \cite{fg09}. This was generalized further by Howie and Purcell in \cite{HowiePurcell} for weakly generalized alternating links. We review this generalization in this section.

\subsection{A decomposition of weakly generalized alternating link complements}\label{Sec:ChunkDecomp}
A checkerboard colorable diagram admits two checkerboard surfaces, one white and one shaded. 
After choosing a checkerboard coloring of the diagram, the white checkerboard colored surface is obtained by taking a surface corresponding to each white region of the diagram, and connecting these surfaces by twisted bands at crossings. Thus it embeds in the (open) link complement. Its intersection with the (compact) link exterior is a properly embedded surface with boundary on $\bdy N(L)$. 
Similarly for the shaded surface. The decomposition of weakly generalized alternating link complements is obtained by cutting along the white and shaded checkerboard surfaces. 
Precisely, remove an open regular neighborhood of the checkerboard surfaces from the link exterior.
Since the regions of these surfaces tile $\Pi$, this cuts $Y-N(L)$ into components corresponding to $Y-N(\Pi)$, with $\bdy N(\Pi)$ decorated with portions of checkerboard surfaces and remnants of $\bdy N(L)$.
The checkerboard surfaces intersect exactly at crossing arcs. After cutting, each crossing arc gives rise to four ideal edges lying on $\bdy(N(\Pi))$, two in each of $\Pi^{\pm}$. Strands of the knot corresponding to overcrossings (resp.\ undercrossings) become ideal vertices in $\Pi^+$ (resp.\ in $\Pi^-$); we contract each of these to lie at a crossing of the diagram. 
More details are in \cite{HowiePurcell}.

We record here the results of Propositions~3.1 and~3.3 of \cite{HowiePurcell}, which prove that the decomposition has the following properties.
The components of the decomposition, which are called \emph{chunks}, are homeomorphic to connected components of $Y- N(\Pi)$. These are 3-manifolds with boundary, where the boundary components are components of $\bdy Y$, along with $\Pi_i^-$ and $\Pi_i^+$. Note that each $\Pi_i$ lies on the boundary of one or two chunks, appearing as $\Pi_i^+$ and $\Pi_i^-$. When $\Pi$ is connected, we have one or two chunks in total, depending on whether $\Pi$ is separating or not. For example if $\Pi$ is the usual projection sphere $S^2$ for a classical alternating link diagram, there are two chunks, each a ball component of $S^3-N(S^2)$.

Decorate the surfaces $\Pi_i^-$ and $\Pi_i^+$ with:
\begin{enumerate}
\item A copy of the edges of the link diagram on $\Pi_i$. These will be called \emph{interior edges} of a chunk, and correspond to crossing arcs in $Y-N(L)$. There will be four such ideal edges in $\bdy N(\Pi)$ per crossing arc, two in each of $\Pi^{\pm}$.
\item Interior edges meet at crossings of the diagram, which become \emph{ideal vertices}. That is, crossings become vertices, which are removed (ideal).
\item Regions bounded by edges and vertices will be called \emph{faces} of the chunk. They are not necessarily simply connected. The faces correspond to regions of the diagram, and will be checkerboard colored. For example, the outer region of \reffig{CheckColorable} is an annulus, so gives rise to a face that is an annulus.
\end{enumerate}
An example of the chunk decomposition for the link from \reffig{CheckColorable}, right, is shown in \reffig{ChunkDecomp1}. The two chunks are the components of the complement of $T^2\times(-\epsilon, \epsilon)$ in $T^2\times[-1,1]$ for some $\epsilon>0$. Hence they are homeomorphic to $T^2\times[-1, -\epsilon]$ and $T^2\times[\epsilon,1]$. Edges, faces, and ideal vertices are marked on $T^2\times\{-\epsilon\}$ and $T^2\times\{\epsilon\}$, respectively, with faces shown checkerboard colored. 

\begin{figure}[h]
  \includegraphics{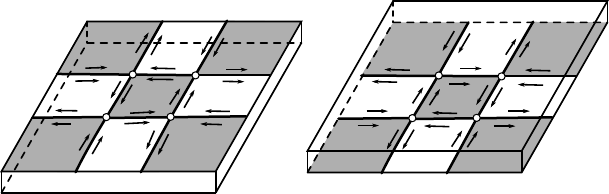}
  \caption{An example of a chunk decomposition. On the left is a manifold homeomorphic to $T^2\times[-1,-\epsilon]$ with faces, edges, and ideal vertices marked on $T^2\times\{-\epsilon\}$. On the right is a manifold homeomorphic to $T^2\times[\epsilon,1]$ with faces, edges, and ideal vertices marked on $T^2\times\{+\epsilon\}$.}
  \label{Fig:ChunkDecomp1}
\end{figure}

Associated to a weakly generalized alternating link is a chunk decomposition corresponding to a weakly generalized alternating diagram constructed as above. We will assume throughout that this is the chunk decomposition used in this paper.

\subsection{Gluing}\label{Sec:Gluing}
To obtain $Y- L$ from the chunks, each face of $\Pi_i^-$ should be glued to the corresponding face of $\Pi_i^+$. A face $F^-$ of $\Pi_i^-$ corresponds to a face $F^+$ of $\Pi_i^+$ if the same region of $\pi(L)$ gave rise to $F^-$ and $F^+$.
The gluing is by the identity away from boundary components of the faces, i.e.\ away from the edges of the chunk(s). In a neighborhood of a boundary component of the face, the gluing rotates the ideal edges by one notch either clockwise or counterclockwise, depending on whether the face is white or shaded. Namely, when viewing $\Pi_i^{\pm}$ from $Y-N(\Pi)$ near $\Pi_i^+$, an ideal edge in a white face will rotate to the next ideal edge of the same face in the clockwise direction for the gluing, and similarly for shaded faces in the counterclockwise direction. The arrows on the faces of the chunk in \reffig{ChunkDecomp1} indicate the gluing.

Under the gluing, four interior edges glue to a single \emph{crossing arc} in $Y- L$.
The crossing arc is identified to two edges each on $\Pi^-$ and $\Pi^+$, with the edges meeting as opposite edges at a vertex. This is illustrated in \reffig{CrossingArc}.

\begin{figure}
\begingroup%
  \makeatletter%
  \providecommand\color[2][]{%
    \errmessage{(Inkscape) Color is used for the text in Inkscape, but the package 'color.sty' is not loaded}%
    \renewcommand\color[2][]{}%
  }%
  \providecommand\transparent[1]{%
    \errmessage{(Inkscape) Transparency is used (non-zero) for the text in Inkscape, but the package 'transparent.sty' is not loaded}%
    \renewcommand\transparent[1]{}%
  }%
  \providecommand\rotatebox[2]{#2}%
  \newcommand*\fsize{\dimexpr\f@size pt\relax}%
  \newcommand*\lineheight[1]{\fontsize{\fsize}{#1\fsize}\selectfont}%
  \ifx\svgwidth\undefined%
    \setlength{\unitlength}{216.19116211bp}%
    \ifx\svgscale\undefined%
      \relax%
    \else%
      \setlength{\unitlength}{\unitlength * \real{\svgscale}}%
    \fi%
  \else%
    \setlength{\unitlength}{\svgwidth}%
  \fi%
  \global\let\svgwidth\undefined%
  \global\let\svgscale\undefined%
  \makeatother%
  \begin{picture}(1,0.52271291)%
    \lineheight{1}%
    \setlength\tabcolsep{0pt}%
    \put(0,0){\includegraphics[width=\unitlength,page=1]{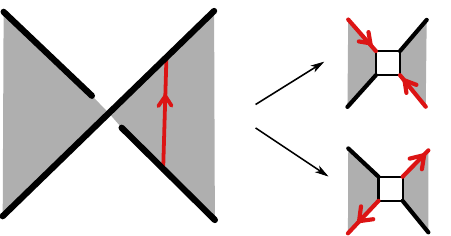}}%
    \put(0.96133115,0.46164811){\color[rgb]{0,0,0}\makebox(0,0)[lt]{\lineheight{1.25}\smash{\begin{tabular}[t]{l}$\Pi^+$\end{tabular}}}}%
    \put(0.97204436,0.00026956){\color[rgb]{0,0,0}\makebox(0,0)[lt]{\lineheight{1.25}\smash{\begin{tabular}[t]{l}$\Pi^-$\end{tabular}}}}%
  \end{picture}%
\endgroup%

  \caption{Four edges are identified to a crossing arc, with two on each of $\Pi^-$ and $\Pi^+$. The edges meet as opposite edges at a vertex.}
  \label{Fig:CrossingArc}
\end{figure}

Now truncate the ideal vertices of chunks: this replaces an ideal vertex with a quadrilateral \emph{truncation face}. We call an edge bordering a truncation face a \emph{truncation edge}. The truncation faces correspond to crossings of the link diagram, and tile the boundary torus $\bdy N(L)$ with a \emph{harlequin tiling}. See \reffig{HarlequinTiling}.
Note the slightly different terminology here: such faces and edges are called boundary faces and edges in \cite{HowiePurcell}.

\begin{figure}[h]
\includegraphics{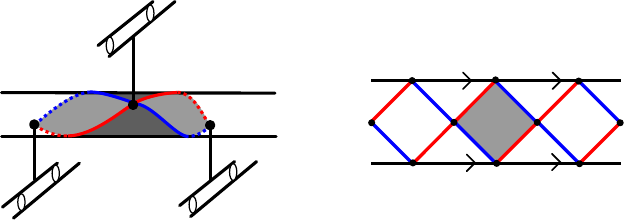}
  \caption{On the left is shown a single truncation face (shaded) as it appears embedded in the link complement. On the right, the boundary torus of the link has been unrolled into an annulus; truncation faces are shown.}
  \label{Fig:HarlequinTiling}
\end{figure}

A chunk with truncated ideal vertices is called a \emph{truncated chunk}. Below, if we refer to faces of a chunk, we mean both truncation faces and other faces. Similarly, if we refer to an edge of a chunk, we mean either a truncation edge or an interior edge. We note that under the gluing described above, the truncation faces that are adjacent to an interior face $I$ are also rotated one notch around $I$, either clockwise or counterclockwise, together with the boundary of $\partial I$.

\section{Normal surfaces}\label{Sec:NormalSfc}

Both normal and standard position begin with a surface being transverse: normal surfaces are transverse to faces, edges, and vertices of tetrahedra in a triangulation of a 3-manifold, and surfaces in standard position are transverse to a projection surface of a link away from crossings. Here we review surfaces that are normal with respect to a chunk decomposition, introduced in \cite{HowiePurcell}, which generalizes both of the above.

\begin{assumption}
We use the topological definitions of incompressible, boundary incompressible, and essential surfaces:
\begin{itemize}
\item A surface $Z$ that is neither a disk nor a 2-sphere is incompressible in a 3-manifold $M$ if, for any embedded disk $D\subset M$ with $D\cap Z = \bdy D$, $\bdy D$ bounds a disk in $Z$. A surface that is not incompressible admits an essential compression disk, namely an essential disk $D\subset M$ with $D\cap Z=\bdy Z$, for which $\bdy D$ is an essential simple closed curve on $Z$.
\item A boundary compression disk for a properly embedded non-disk surface $Z$ in $M$ is an embedded disk $D\subset M$ with $\bdy D$ consisting of two arcs: $\alpha = D\cap Z$ and $\beta=D\cap \bdy M$, with $\alpha \cap \beta = \bdy\alpha = \bdy \beta$. The boundary compression disk is essential if $\alpha$ does not cobound a disk in $Z$ with another arc in $\bdy Z$. If there exists an essential boundary compression disk for $Z$, then $Z$ is boundary compressible. Otherwise, it is boundary incompressible. 
\item A 2-sphere is incompressible if and only if it does not bound a 3-ball.
\item By convention, disks will be neither incompressible nor compressible, and neither boundary incompressible nor boundary compressible in this paper. 
\item A surface is essential if it is incompressible, boundary incompressible, and not boundary parallel.
\item Note we allow $Z$ to be orientable or nonorientable, with or without boundary.
\end{itemize}
\end{assumption}

\subsection{Normal surface in a chunk}

First, consider a surface $Z'$, possibly with boundary, properly embedded in a truncated chunk $C$, with $\bdy Z'\subset \bdy C$.

\begin{definition}\label{Def:NormalSurface}
The surface $Z'$ is \emph{normal} with respect to the chunk $C$ if it satisfies the following.
\begin{enumerate}
\item[(0)] Each non-disk component of $Z'$ is incompressible in $C$.
\item[(1)] $Z'$ and $\bdy Z'$ are transverse to all faces, edges, and vertices of $C$.
\item[(2)] If a component of $\bdy Z'$ lies entirely in a face of $C$, then it does not bound a disk in that face.
\item[(3)] If an arc $\gamma$ of $\bdy Z'$ in a face of $C$ has both endpoints on the same edge, then the arc $\gamma$ along with an arc of the edge cannot bound a disk in that face.
\item[(4)] If an arc $\gamma$ of $\bdy Z'$ in an interior face of $C$ has one endpoint on a truncation edge and the other on an adjacent interior edge, then
the union of $\gamma$ as well as adjacent arcs of the two edges cannot bound a disk in the interior face.
\end{enumerate}
\end{definition}

Example of arcs that make a surface fail to be normal are shown in \reffig{NonNormal}.

\begin{figure}[h]
  \includegraphics{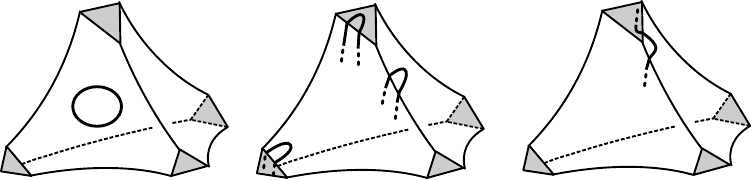}
  \caption{Shown are arcs of $\bdy Z'\cap \bdy C$ that make the surface $Z'$ fail to be normal. On the left, the surface fails (2), in the middle it fails (3), on the right it fails (4).}
  \label{Fig:NonNormal}
\end{figure}

\subsection{Normal surface with respect to a chunk decomposition}
Given a chunk decomposition of $Y- L$, a surface $Z$ embedded in $Y- L$ is \emph{normal} with respect to the chunk decomposition if for every chunk $C$, the intersection $Z\cap C$ is a (possibly disconnected) normal surface in $C$.

Here and further we subdivide a surface $Z$ in subsurfaces $Z_i$ cut out by faces of the chunks. We assume that each $Z_i$ is connected, closed or with boundary, and possibly with multiple boundary components. All our surfaces and subsurfaces from now on will be embedded. Moreover, we will assume that if the surface $Z$ has boundary and $Y$ has boundary, then the boundary of $Z$ lies on $N(L)$ and is disjoint from $\bdy Y$. This includes surfaces of any slope on $L$, as well as closed surfaces.

It was shown in Theorem~3.8 of \cite{HowiePurcell} that any essential surface in a 3-manifold with a chunk decomposition can be put into normal form with respect to the chunk decomposition; see also \cite[Theorem~2.8]{fg09}. 
We will revisit this proof below in \refthm{Normal} to show that the process of putting surfaces into normal form preserves other desirable features of the surface.

We note that item~(0) of \refdef{NormalSurface} is slightly stronger than the definition in \cite[Definition~3.7]{HowiePurcell}, in that we require all non-disk components $Z_i$, and not just closed components, to be incompressible. However, we will see in \refthm{Normal} below that any surface that is normal with respect to the definition of \cite{HowiePurcell} can be isotoped to be normal with respect to \refdef{NormalSurface}.

\section{Saddles, meridian punctures, and boundary of the surface}\label{Sec:SBP}

For normal subsurfaces $Z_j$ of the surface $Z$, each boundary component of $\bdy Z_j$ runs over truncation edges and interior edges of the respective chunk. In this section we will describe a labeling of the components of $\bdy Z_j$ by letters $P$, $S$, and $B$, analogous to similar labelings in \cite{Menasco1984, HTT1, HTT2}.

In these other papers, an essential surface $Z$ in an alternating link complement in $S^3$ is studied by considering its curves of intersection with the projection sphere and with small balls around crossings called \emph{crossing balls}.
The surface in standard position from \cite{Menasco1984, HTT1, HTT2} intersects the projection sphere where $Z$ has \emph{saddles} inside crossing balls. See \reffig{Saddle} left.
In \cite{Menasco1984, HTT1, HTT2}, each saddle is labeled with the letter $S$.

For a weakly generalized alternating link, if a surface $Z$ has a saddle, that saddle intersects a crossing arc exactly once. Recall from subsection~\ref{Sec:Gluing} that each interior edge of the chunk decomposition is identified to exactly one crossing arc, with four interior edges in total identified to a single crossing arc. Thus saddles meet interior edges as in \reffig{Saddle}, right. Shown in that figure are the four edges that glue up to the crossing arc (red), and the way the saddle surface meets them.
For a surface in normal form, we consider how the surface intersects interior edges and truncation edges. Because interior edges are identified to crossing arcs, and each such arc runs through exactly one saddle, analogous to the notation of \cite{Menasco1984}, we label each intersection of $\bdy Z_j$ and an interior edge with an $S$.

\begin{figure}[h]
  \includegraphics{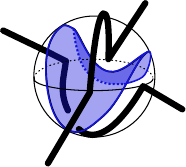}
  \hspace{.5in}
\begingroup%
  \makeatletter%
  \providecommand\color[2][]{%
    \errmessage{(Inkscape) Color is used for the text in Inkscape, but the package 'color.sty' is not loaded}%
    \renewcommand\color[2][]{}%
  }%
  \providecommand\transparent[1]{%
    \errmessage{(Inkscape) Transparency is used (non-zero) for the text in Inkscape, but the package 'transparent.sty' is not loaded}%
    \renewcommand\transparent[1]{}%
  }%
  \providecommand\rotatebox[2]{#2}%
  \newcommand*\fsize{\dimexpr\f@size pt\relax}%
  \newcommand*\lineheight[1]{\fontsize{\fsize}{#1\fsize}\selectfont}%
  \ifx\svgwidth\undefined%
    \setlength{\unitlength}{216.19116211bp}%
    \ifx\svgscale\undefined%
      \relax%
    \else%
      \setlength{\unitlength}{\unitlength * \real{\svgscale}}%
    \fi%
  \else%
    \setlength{\unitlength}{\svgwidth}%
  \fi%
  \global\let\svgwidth\undefined%
  \global\let\svgscale\undefined%
  \makeatother%
  \begin{picture}(1,0.52271291)%
    \lineheight{1}%
    \setlength\tabcolsep{0pt}%
    \put(0,0){\includegraphics[width=\unitlength,page=1]{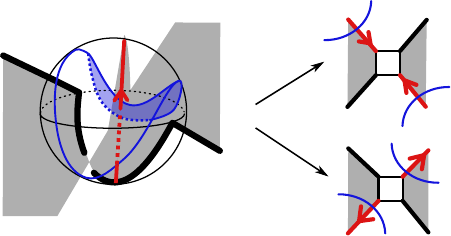}}%
    \put(0.96133108,0.46164811){\color[rgb]{0,0,0}\makebox(0,0)[lt]{\lineheight{1.25}\smash{\begin{tabular}[t]{l}$\Pi^+$\end{tabular}}}}%
    \put(0.97204428,0.0002696){\color[rgb]{0,0,0}\makebox(0,0)[lt]{\lineheight{1.25}\smash{\begin{tabular}[t]{l}$\Pi^-$\end{tabular}}}}%
    \put(0,0){\includegraphics[width=\unitlength,page=2]{SaddleChunk.pdf}}%
  \end{picture}%
\endgroup%

\caption{Left: A saddle runs between arcs of the diagram at a crossing. Right: This leads to $Z$ intersecting interior edges of the chunk diagram. Shown are intersections of interior edges on $\Pi^+$ and $\Pi^-$.}
\label{Fig:Saddle}
\end{figure}

Recall that a surface $Z$ in a link exterior $Y-N(L)$ is \emph{meridianally compressible} if there is a disk $D$ embedded in $Y$ such that $D\cap Z = \bdy D$, the interior of $D$ intersects $L$ exactly once transversely, and $\bdy D$ is not parallel through (an annulus in) $Z$ to a meridian of $\bdy N(L)$. Such a disk $D$ is called a meridianal compression disk.
Otherwise $Z$ is \emph{meridianally incompressible}. For a meridianally compressible surface, performing surgery along $D-N(L)$ yields a new surface whose boundary is the union of $\bdy Z$ and two meridians on $\bdy N(L)$; this is called a \emph{meridianal compression} of $Z$.

Both for classical alternating links in $S^3$~\cite{Menasco1984} and for weakly generalized alternating links under certain conditions~\cite[Lemma~4.9]{HowiePurcell}, every closed surface $Z'$ is meridianally compressible. After meridianal compressions the resulting surface $Z$ meets the diagram in meridians. In \cite{Menasco1984, HTT1} for links in $S^3$, each meridian on $\bdy Z$ is labeled with a $P$, which stands for a meridianal \textit{puncture}. For weakly generalized alternating links, when $\bdy Z$ again consists only of meridians, each meridian will intersect two truncation faces. We also wish to label intersections with a $P$,
but the setup will be slightly different. To describe our labelling, we first need the following definition.

\begin{definition}\label{Def:MeridianalForm}
Isotope $Z$ to meet the diagram $\pi(L)$ transversely away from crossings.
After this isotopy, in the chunk decomposition each meridianal curve of $\bdy Z$ meets exactly two truncation faces, which are quadrilaterals. It runs through adjacent truncation edges on each face, cutting off a single corner of each of two quads in the harlequin tiling of the boundary; see \reffig{MeridianalForm}. Moreover, note that two opposite vertices of a truncation face are always identified (they lie on the same endpoint of a crossing arc). The meridian does not cut off such a vertex, but rather cuts off one of the other two vertices of a truncation face. We say that a component of $\bdy Z$ is in \emph{meridianal form} if each component of $\bdy Z$
\begin{enumerate}
\item meets exactly two truncation faces, one on each side of the projection surface;
\item runs between adjacent edges in each such truncation face, cutting off a vertex that is not identified to another corner on the same truncation face.
\end{enumerate}
\noindent See \reffig{MeridianalForm}.
We say that $Z$ is in meridianal form if each component of $\bdy Z$ is in meridianal form.
\end{definition}

\begin{figure}[h]
\begingroup%
  \makeatletter%
  \providecommand\color[2][]{%
    \errmessage{(Inkscape) Color is used for the text in Inkscape, but the package 'color.sty' is not loaded}%
    \renewcommand\color[2][]{}%
  }%
  \providecommand\transparent[1]{%
    \errmessage{(Inkscape) Transparency is used (non-zero) for the text in Inkscape, but the package 'transparent.sty' is not loaded}%
    \renewcommand\transparent[1]{}%
  }%
  \providecommand\rotatebox[2]{#2}%
  \newcommand*\fsize{\dimexpr\f@size pt\relax}%
  \newcommand*\lineheight[1]{\fontsize{\fsize}{#1\fsize}\selectfont}%
  \ifx\svgwidth\undefined%
    \setlength{\unitlength}{360bp}%
    \ifx\svgscale\undefined%
      \relax%
    \else%
      \setlength{\unitlength}{\unitlength * \real{\svgscale}}%
    \fi%
  \else%
    \setlength{\unitlength}{\svgwidth}%
  \fi%
  \global\let\svgwidth\undefined%
  \global\let\svgscale\undefined%
  \makeatother%
  \begin{picture}(1,0.292892)%
    \lineheight{1}%
    \setlength\tabcolsep{0pt}%
    \put(0,0){\includegraphics[width=\unitlength,page=1]{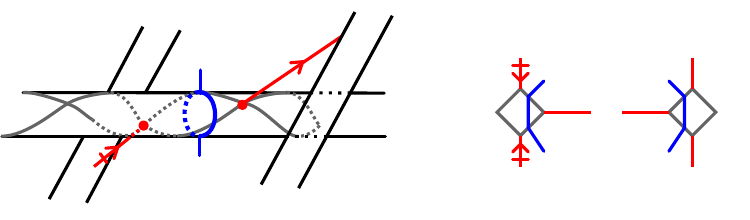}}%
    \put(0.75669826,0.16679205){\color[rgb]{0,0,0}\makebox(0,0)[lt]{\lineheight{1.25}\smash{\begin{tabular}[t]{l}$\Pi^+$\end{tabular}}}}%
    \put(0,0){\includegraphics[width=\unitlength,page=2]{MeridinalForm.pdf}}%
    \put(0.96561298,0.16568718){\color[rgb]{0,0,0}\makebox(0,0)[lt]{\lineheight{1.25}\smash{\begin{tabular}[t]{l}$\Pi^-$\end{tabular}}}}%
  \end{picture}%
\endgroup%

  \caption{Left: Isotope $Z$ with meridianal boundary to meet $N(L)$ transversely away from crossings. Right: In the chunk decomposition, such a curve meets exactly two truncation faces, and cuts off a single corner of each truncation face. Note that the corner cut off by a meridian is not one of the corners that are identified.
  }
  \label{Fig:MeridianalForm}
\end{figure}

Suppose $Z$ is a surface in meridianal form. Assign a label $P$ to each intersection of $\bdy Z$ with a truncation face.
Note just as for a classical alternating link in $S^3$, a meridianal puncture in $Z$ corresponds to a single $P$ above the diagram on $\Pi^+$, and a single $P$ below on $\Pi^-$.

For a spanning surface $Z$ of an alternating link in $S^3$, Hass, Thompson, and Tsvietkova~\cite{HTT2} introduce the letters $B$. A letter $B$ indicates where $Z$ intersects the link transversally on the projection sphere.
In the case of weakly generalized alternating links, a surface $Z$ with boundary will have normal subsurfaces $Z_j$ with $\bdy Z_j$ meeting truncation faces. When the corresponding curve on $\bdy Z$ is not necessarily a meridianal curve in meridianal form, we label each intersection of $\bdy Z_j$ with a truncation edge by $B$. Note that above, we labeled intersections of $\bdy Z_j$ and truncation faces by $P$. For surfaces with non-meridianal boundary, we consider intersections with truncation edges rather than faces. The letter $B$ is a reminder that the surface might have boundary components that are not meridianal.

The labeling of components of $\bdy Z_i$ (which are curves) by letters $S$, $B$, or $P$ associates a cyclic word to every such curve. An example is shown in \reffig{SBPCurves}.

\begin{figure}[h]
  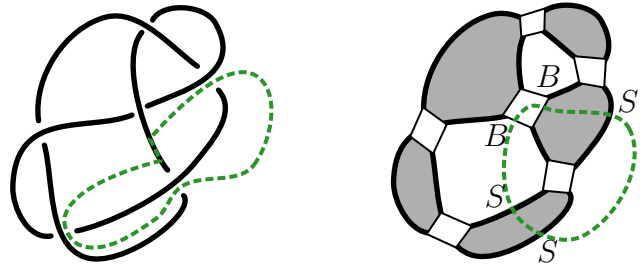
  \caption{On the left is shown an example of a $BBSSS$ curve from \cite{HTT2}. On the right is shown the corresponding $BBSSS$ curve in a chunk of the decomposition of the same link. The curve on the right meets the same saddles (interior edges) and link overcrossings (truncation faces) as the curve on the left.}
  \label{Fig:SBPCurves}
\end{figure}

\begin{remark}\label{Rem:BBReplacesP}
  In fact, the arguments below for surfaces with labels $B$ on truncation edges work equally well when we replace an instance of $P$ with two instances of $B$, and so we allow curves with $B$ labels also to be in meridianal form. This will be useful for links, for example, when some components of $\bdy Z$ are meridians and some are not. In this mixed case, with some meridian boundary components and some non-meridians, we will label all intersections with truncation edges by $B$. However, when a surface has strictly meridianal boundary, we obtain more information using only labels $P$. We do not allow words in both $P$ and $B$.
\end{remark}

For classical alternating links, the curves of intersection subdivide the surface into disks lying in topological 3-balls above and below the link. The curves of intersection and the position of the disks are uniquely determined by the associated words in $S$, $P$, and $B$ and by the position of each letter on the link diagram. The disks are then glued along their boundaries, and the gluing pattern is also uniquely determined by the these words in $S$, $P$ and $B$ and by the position of each letter on the link diagram. For weakly generalized alternating links, there are also words associated to boundary components of $Z_i$, and the position of each letter on the knot projection. But the boundary curves subdivide the surface $Z$ into subsurfaces $Z_i$ that might have positive genus, and multiple boundary components. Moreover, the subsurfaces are not in topological 3-balls anymore: rather, they are in chunks which might have complicated topology themselves.


\section{Normal form, saddles, and meridians}\label{Sec:Meridianal}

For a surface $Z$, associate a pair of natural numbers $(s(Z), t(Z))$ where $s(Z)$ is the number of intersections of $Z$ with interior edges, and $t(Z)$ is the number of intersections with truncation edges. Call the pair $(s(Z), t(Z))$ the \emph{weight} of the embedding with respect to a chunk decomposition, and order lexicographically.
In this section, we ensure that a surface can be isotoped into normal form while preserving meridianal form for the surface, by an isotopy that does not increase weight. Note that there are many (isotopic) ways to put a surface in normal form or standard position, and in applications, often the one that minimizes weight is chosen.

For convenience, we collect assumptions here. They will be used throughout the rest of the paper, and will be part of the hypothesis in each lemma, proposition and theorem we prove.

\begin{assumption}
\begin{itemize}
\item Let $Y$ and $\Pi$ satisfy Assumption~\ref{Assumption:3mfld}. We assume that $\pi(L)$ is a weakly generalized alternating projection of a link $L$ onto $\Pi$ in $Y$.
\item We consider $(Z,\bdy Z)$ to be an essential surface (either orientable or nonorientable) embedded in $(Y-N(L), \bdy N(L))$, where the boundary of $Z$ is possibly empty. Note the boundary of $Z$ (if any) is disjoint from $\bdy Y$.

\item We also assume that there is a chunk decomposition of $Y-N(L)$ arising from a weakly generalized alternating diagram as in subsection~\ref{Sec:ChunkDecomp}. The $Z_i$ are the connected subsurfaces of $Z$ cut out by chunks, each closed or with boundary, and possibly with multiple boundary components. 
\end{itemize}
\end{assumption}

\begin{theorem}\label{Thm:Normal}
A surface $Z$ in $Y-N(L)$ can be isotoped into normal form such that:
\begin{enumerate}
\item[(a)] The full isotopy consists of a sequence of sub-isotopies given by discrete steps, and no step of the isotopy increases the weight. 
\item[(b)] If $Z$ begins in meridianal form, then after the isotopy into normal form, $Z$ remains in meridianal form.
\end{enumerate}
\end{theorem}

\begin{proof}
The proof that an isotopy exists putting the surface into normal form is a standard innermost disk / outermost arc argument that is very similar to arguments that appear elsewhere (see \cite{HowiePurcell, fg09}), but we walk through it to verify (a) and (b). 

We need to check the requirements of \refdef{NormalSurface}. 
A small isotopy of $Z$ ensures transversality conditions~(1). We may ensure this isotopy does not increase the number of intersections with any edge, for~(a), and does not affect meridianal form for~(b).

Since $Z$ is essential, condition~(0) automatically holds for closed surfaces in any chunk, without any isotopy. If $D$ is an essential compressing disk for $Z\cap C$ within a chunk $C$, then because $Z$ is incompressible, $\bdy D$ bounds a disk $E$ in $Z$, and by irreducibility of $Y$, $D\cup E$ bounds a ball. Hence $E$ can be isotoped through the ball and past $E$ to remove the compressing disk. For~(a), we say that this full isotopy through the ball is one step. Note the weight may increase temporarily during the isotopy through the ball, but when this step of the isotopy is completed, the number of intersections of $Z$ with interior edges has not increased. The isotopy does not affect the boundary of $Z$ at all, completing~(a) and giving~(b).

For condition~(2) of \refdef{NormalSurface}, suppose some $Z_k$ has boundary $\bdy Z_k$ lying in a face of $C$ and bounding a disk in that face. Take an innermost such disk $D$. It is not an essential compressing disk for $Z$ in $Y$ because $Z$ is incompressible. Because $Y$ is irreducible, there is a ball $B$ with $D\subset \bdy B$ and $\bdy B-D\subset Z$. Isotope $Z$ through $B$ and slightly further, removing all intersections of $Z$ with faces and interior edges that lie within the ball $B$, and removing the intersection of $\bdy Z_k$ on a face of $C$, all without changing $Z$ outside a small neighborhood of $B$. Again the full isotopy through $B$ is considered to be one step, for~(a). Note this move cannot increase the number of intersections of $Z$ with interior edges, and avoids the boundary of $Z$ entirely, so it cannot increase the weight for~(a), and does not affect meridianal form for~(b).

For condition~(3) of \refdef{NormalSurface}, suppose for some $Z_k$, an arc of $\bdy Z_k$ lies in a single face with endpoints on the same edge, and together with a part of the edge cuts off a disk $D$ in that face. Assume $D$ is innermost with this property in that face, meaning its interior does not intersect $Z$.

There are three cases depending on the type of edge and the type of face.

Suppose first that the arc of $\bdy Z_k$ has endpoints on the same interior edge. Then use a regular neighborhood of the disk $D$ to isotope $Z$ past that edge (a step of the isotopy), strictly reducing the number of intersections with the interior edge, and not affecting $\bdy Z$, giving~(a) and~(b).

Next suppose that the arc of $\bdy Z_k$ lies in a truncation face, with both endpoints on the same truncation edge. Then again use a regular neighborhood of $D$ in $Y-N(L)$ to isotope $\bdy Z$ along this truncation face and past the edge, removing two intersections with truncation edges. This move is a step of the full isotopy. It does not affect intersections with interior edges, and strictly decreases intersections with truncation edges, giving~(a). Such an arc is not in meridianal form, so does not arise for~(b).

Finally suppose the arc $\alpha$ of $\bdy Z_k$ has both endpoints on the same truncation edge but lies in an interior face. Then the disk $D$ has the form of a boundary compression disk for $Z$. Since $Z$ is boundary incompressible, there must be an arc $\beta\subset \bdy Z$ whose endpoints agree with those of $\alpha$, and such that $\alpha\cup \beta$ bounds a disk $D'$ in $Z$. Then $D\cup D'$ forms a disk with boundary on $\bdy N(L)$. Because $Y- N(L)$ is boundary irreducible, by \cite[Corollary~3.16]{HowiePurcell}, the disk $D\cup D'$ must co-bound
a ball with a disk of $\bdy N(L)$. Use this ball to isotope $D'$ in $Z$ past $D$ to remove the two intersections with the truncation face, as well as any other intersections of $Z$ with edges and faces within the ball.
For~(a), the isotopy through this ball is a step, and when complete, the number of intersections with interior and truncation edges does not increase. As for~(b), the arc $\beta$ of $\bdy Z$ cannot be in meridianal form, since both its endpoints are on the same truncation edge, so this situation does not arise for~(b). 

Condition~(4) of \refdef{NormalSurface} requires the most care, since it could arise for surfaces in meridianal form. Nevertheless, we show that an isotopy also addresses this case while preserving meridianal form and without increasing weight. The argument requires careful consideration of the combinatorics of the chunk decomposition around a truncation face and an adjacent interior edge. We step through it below, using a number of figures.

Suppose an arc $\gamma$ of $\bdy Z_k$ in an interior face of a chunk $C$ has an endpoint on a truncation edge and an endpoint on an adjacent interior edge, cutting off a disk $D$ with these two edges. 
We may take $\gamma$ to be outermost with this property, so that the interior of $D$ is disjoint from $Z$.
Isotope $Z$ through a regular neighborhood of $D$, by sliding (a neighborhood of an arc of) $\bdy Z$ along the adjacent truncation faces and past the interior edge, then sliding the rest of a neighborhood of $\gamma$ in $Z$ through a neighborhood of $D$ to follow. The effect of this move on the harlequin tiling of the boundary is shown in \reffig{Condition4-BdryView}.
This removes an intersection of $Z$ with an interior edge, giving~(a). This will introduce an arc in a truncation face with both endpoints on the same truncation edge, as in \reffig{Condition4-BdryView}. Such an arc can be eliminated as above, reducing weight. Thus if the surface does not begin in meridianal form, we are done at this point.

\begin{figure}
  \includegraphics{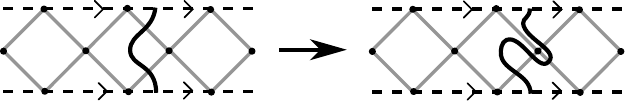}
  \caption{A curve of $\bdy Z$ forming a meridian on the harlequin tiling of the boundary is shown on the left. If we isotope across a disk that violates condition~(4), the curve is changed as shown on the right. Note further isotopy is required to put it into normal form, meridianal form.}
  \label{Fig:Condition4-BdryView}
\end{figure}

To complete the proof of~(b), we need to analyse the isotopy of $Z$ through $N(D)$ more carefully, ensuring that if $D$ is adjacent to a truncation face that $Z$ meets in meridianal form, then after the isotopy the result is still in meridianal form.

Suppose $Z$ is in meridianal form. We set up some notation. Assume that $D\subset \Pi^+$. The arc $\gamma$ of $\bdy Z_k$ that cuts off $D$ continues through the truncation face in meridianal form. Further, because the interior face containing $D$ is glued to another interior face on $\Pi^-$, there is another such arc cutting off a disk $D'\subset \Pi^-$, and that arc extends to run through another truncation face in meridianal form. Indeed, these two arcs in meridianal form together form a meridian boundary component of $\bdy Z$. Finally, on each of $\Pi^+$ and $\Pi^-$ there is another interior edge identified to the one meeting $D$ or $D'$. The surface $Z$ must also run through that interior edge. The setup must therefore appear as in \reffig{NormalMerid}.

\begin{figure}[h]
  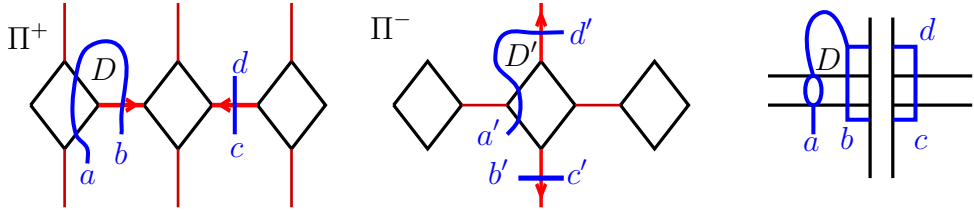
  \caption{When $Z$ is in meridianal form, and an arc runs from a truncation edge to an adjacent interior edge, $Z$ must intersect $\Pi^+$ and $\Pi^-$ as shown on the left (possibly with roles of $\Pi^+$, $\Pi^-$ switched). The right shows the saddle and the adjacent meridian in $Y-N(L)$.}
  \label{Fig:NormalMerid}
\end{figure}

That is, there are arcs $ab$ and $cd$ of $Z_k\cap \Pi^+$ meeting interior edges identified to the same crossing arc of a saddle, and a subarc of $ab$ is an arc of the boundary of the disk $D$. There are additional arcs $a'd'$ and $b'c'$ in $\Pi^-$, also meeting interior edges identified to the same crossing arc, with an arc of $a'd'$ forming an arc of the boundary of disk $D'$. Moreover, $a$ is glued to $a'$, $b$ to $b'$, $c$ to $c'$ and $d$ to $d'$.

For the leftmost and middle pictures in \reffig{NormalMerid}, note that the endpoints of arcs labeled by $a, b,c,d$ are in the same quadrants as $a', b', c', d'$ respectively. This is because two interior faces that are identified correspond to the same region of the link diagram. Also note that the arcs $ab, cd, a'd', b'c'$ and their intersection pattern with interior edges and truncation edges of the chunks correspond to rotating the boundary of every interior face under the gluing one notch, as described in \refsec{Gluing}.

From the middle picture in \reffig{NormalMerid}, rotate quadrant I one notch counterclockwise and identify it with quadrant I in the leftmost picture. Similarly, rotate quadrant III one notch counterclockwise and identify it with quadrant III in the leftmost picture. Here interior edges are rotating to the next interior edges in the counterclockwise direction, and truncation edges are rotating to the next truncation edge in the counterclockwise direction. Rotate quadrants II and IV one notch clockwise to the leftmost picture.

Now we isotope $Z$ across the disk $D$, to remove intersections with the interior edge identified to the crossing arc. Note this simultaneously isotopes across $D'$ in the other chunk. This isotopes arcs $cd$ and $b'c'$ to run through the adjacent truncation faces; this is shown in \reffig{NormalMeridIsotope1}. Away from these arcs, the intersection of the surface $Z$ with $\Pi^+$ and $\Pi^-$ is unchanged.

\begin{figure}[h]
  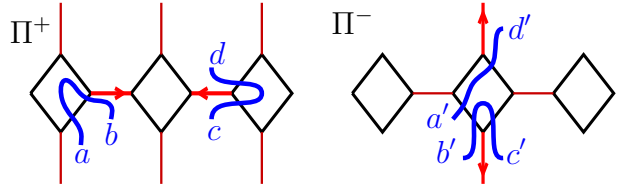
  \caption{The isotopy through $D$ and $D'$ has the effect shown.}
  \label{Fig:NormalMeridIsotope1}
\end{figure}

Note that the surface is not in normal form; on the left of \reffig{NormalMeridIsotope1}, an arc of $\bdy Z'$ in the truncation face meets the same truncation edge twice. We perform an isotopy to slide this arc off this truncation face and onto the truncation face on $\Pi^-$, shown in the center on the right of \reffig{NormalMeridIsotope1}. 
The final result of this isotopy, in $\Pi^-$, $\Pi^+$ and in $Y-N(L)$, is shown in \reffig{NormalMeridIsotope2}.

\begin{figure}[h]
  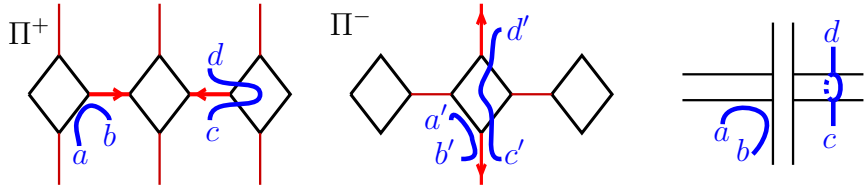
  \caption{Isotoping further into normal form yields surface in meridianal form.
  }
  \label{Fig:NormalMeridIsotope2}
\end{figure}

The surface is now again in meridianal form, for~(b). In the link complement, the entire procedure swept a meridian past a crossing arc, removing the intersection with the corresponding interior edge without introducing new intersections with interior edges, therefore decreasing weight. We call this a step of the isotopy, and note~(a) also holds for this step.
\end{proof}

\begin{assumption}
From now on, when we put a surface in normal position, we always make two assumptions:
\begin{enumerate}
\item we do it so that the surface is in meridianal form;
\item out of all ways to do it so that the surface is in meridianal form, we choose one with least weight.
\end{enumerate}
\end{assumption}

\section{Generalizing standard position}\label{Sec:Standard}

In this section we generalize properties of standard position for closed surfaces from \cite{Menasco1984}, and for surfaces with non-meridianal boundary from Propositions 2.1-2.2 of \cite{MenascoThistlethwaite}.

\begin{proposition}\label{Prop:TruncationEdgesPairs} 
The following holds for a surface $Z$ in $Y-N(L)$, and any connected subsurface $Z_i$ of $Z$ in a chunk $C$:
\begin{enumerate}
\item[(1)] There is no meridianal compression of $Z$ to a component $R$ of $\bdy N(L)$ for which $\bdy Z \cap R$ is nonempty and non-meridianal. 
\item[(2)] $Z_i$ cannot be a sphere or projective plane.
\item[(3)] If $Z$ is in normal position, each curve of $\bdy Z_i$ meets truncation edges in pairs. Thus labels $B$ occur in pairs.
\end{enumerate}
\end{proposition}

\begin{proof}
For (1), suppose $Z$ is a surface with non-meridianal boundary on a component $R$ of $\bdy N(L)$, and suppose that admits a meridianal compression to $R$. Such a compression defines an embedded annulus $A$ with one boundary component on $Z$ and one on $R$, with interior disjoint from $Z$. But a meridian intersects any non-meridianal closed curve on the link boundary torus. If $\bdy Z$ meets this component $R$ of $\bdy N(L)$ and is non-meridianal, $\bdy Z$ must intersect the annulus $A$, and hence $Z$ meets the interior of $A$. This contradicts the definition of $A$.

For (2): Recall that $\Pi$ was chosen such that $Y- \Pi$ is irreducible, thus any chunk is irreducible. Therefore a spherical component cannot bound anything but a ball. But a 2-sphere is incompressible if it does not bound a ball. Hence item (0) of \refdef{NormalSurface} implies that the intersection of a normal surface with a chunk has no spherical components. Similarly, because $Y-\Pi$ is orientable, the boundary of a regular neighborhood of any embedded incompressible projection sphere would be an essential 2-sphere, which cannot exist in the irreducible $Y-\Pi$.

For (3): Every time $\bdy Z_i$ meets a truncation edge, it runs into a truncation face and then must meet another truncation edge as it runs out of that truncation face.
\end{proof}

The next result is analogous to Menasco's result \cite[Lemma~1]{Menasco1984} that no subsurface in standard position meets the same crossing bubble in more than one arc. See also Menasco and Thistlethwaite~\cite[Proposition~2.2(ii)]{MenascoThistlethwaite}. 

Fix a truncation face on a chunk. The face is adjacent to exactly two interior edges, say $e_1$ and $e_2$, that are identified together. Denote by $\gamma$ the arc formed from the union of $e_1, e_2$ and (a choice of) two truncation edges connecting $e_1$ with $e_2$. See \reffig{Cases2Saddles}(a). We call the connected simplex consisting of these four edges a \textit{saddle complex}. 

\begin{proposition}\label{Prop:NoSaddleTwice}
Suppose $Z$ is meridianally incompressible.
Suppose an arc $\alpha$ of $Z\cap \bdy C$ intersects a saddle complex $\gamma$ in exactly two points. Then $\alpha$ cannot co-bound a disk with a subarc $\beta$ of $\gamma$.
\end{proposition}

\begin{proof}
Suppose the arcs cobound a disk. There are six ways that an arc $\alpha$ might co-bound a disk with a subarc of $\gamma$, shown in Figure~\ref{Fig:Cases2Saddles}.
Take $D$ to be an innermost disk with the property that $\bdy D$ consists of an arc $\alpha$ on $Z\cap \bdy C$ and a subarc $\beta$ of $\gamma$. That is, $D\cap Z = \alpha$. 

\begin{figure}[h]
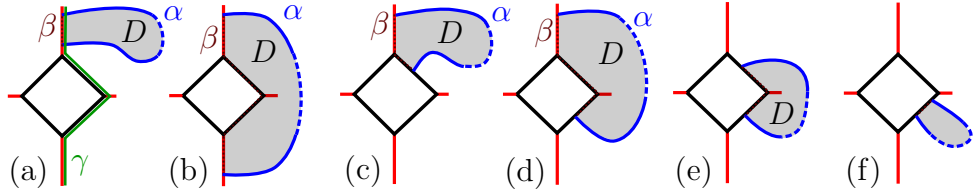
\caption{Shown are the ways that an arc $\alpha$ of $Z\cap \bdy C$ can co-bound a disk $D$ with a subarc $\beta$ of $\gamma$, where $\gamma$ is a saddle simplex.}
\label{Fig:Cases2Saddles}
\end{figure}

Suppose that for $D$ innermost, $\alpha$ has both endpoints on the same interior edge $e$, as in (a) of \reffig{Cases2Saddles}. 
If $\alpha$ happens to lie in only one face, then $Z$ is not normal and we have already seen that $Z$ can be put into normal form in a way that removes $\alpha$ (\refthm{Normal}).
But it could be the case that $D$ intersects multiple faces and edges. In this case,
we will isotope $Z$ to remove two intersections of $Z$ with $e$, illustrated in \reffig{NoSaddleTwice}.

\begin{figure}[h]
  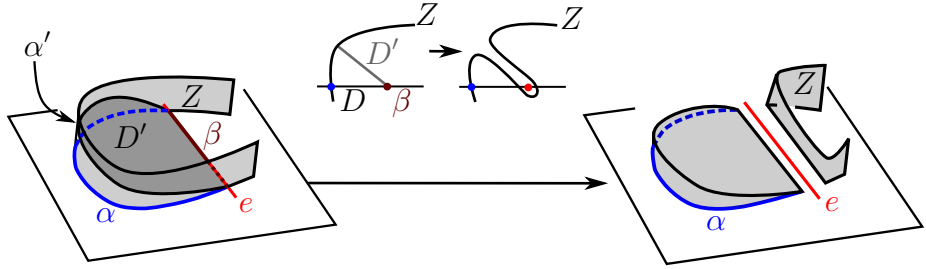
  \caption{Shows how we isotope $Z$ through the disk $D'$ to remove intersections with $e$, in 3-dimensions and a 2-dimensional cross section. In the leftmost and rightmost figures, $e$ is the edge in red, with $\beta$ slightly darker red on $e$. In the center, a cross-section is depicted, and $e$ becomes a single point. }
  \label{Fig:NoSaddleTwice}
\end{figure}

Carefully, push the disk $D$ slightly into the interior of the chunk, obtaining a new disk $D'$, shown in darker gray on the left of \reffig{NoSaddleTwice}. Then isotope $Z$ in a small regular neighborhood of $D'$, removing two intersections with $e$, as shown on the right of the figure. Note that within $Y-N(L)$, the edge $e$ is identified to other edges to obtain a crossing arc that contains $\beta$, and $\beta$ runs between two saddles of $Z$ intersecting the crossing arc. This move eliminates the intersections of $Z$ at endpoints of $\beta$. Note \reffig{NoSaddleTwice} does not show what happens in $Y-N(L)$, only the chunk decomposition, prior to the edge identifications. We call the move of \reffig{NoSaddleTwice} a \emph{band move}, as in~\cite[Figure~2.3(iii)]{MenascoThistlethwaite}. 

Alternatively, the isotopy of a band move is equivalent to the following move. Take a small product neighborhood of $D'$. This is a ball $B=D'\times(-\epsilon,\epsilon)$, for some sufficiently small $\epsilon>0$. The boundary $\bdy B$ consists of a disk $E$ on $Z$ of the form $E=\alpha'\times(-\epsilon,\epsilon)$, and a second disk $F=\bdy B - E$. Replace the surface $Z$ by removing the disk $E$ from $Z$ and replacing it with the disk $F$. Push slightly through the boundary of the chunk. The resulting surface is isotopic to $F$, but meets the edge $e$ two fewer times.

The resulting surface is still essential and meridianally incompressible, and any of its meridian boundary curves are still in meridianal form, so it can be isotoped further into normal form without increasing intersections with interior edges as in \refthm{Normal}. This is a contradiction, since $Z$  was assumed to have minimal weight.

Next suppose that for $D$ innermost, $\alpha$ has endpoints on distinct interior edges of $\gamma$, as in~(b) of \reffig{Cases2Saddles}.
Then glue interior edges to obtain a crossing arc.
Because the disk $D$ is innermost, the two endpoints of $\alpha$ on interior edges are then identified. The arc of $\beta$ on the truncation faces forms a meridian. Thus after gluing, $\alpha\subset Z$ bounds a meridian compression disk. Because $Z$ is meridianally incompressible, $\alpha$ is parallel through an annulus $A \subset Z$ to a meridian component of $\partial Z$, and together $D-N(L)$ and this annulus $A$ cobound a thickened annulus. Isotope $Z$ through this thickened annulus, past $D$, removing all intersections of $Z$ with faces, edges, and truncation edges within the thickened annulus. 
The number of intersections of $Z$ with interior edges strictly decreases, and the meridian component of $\bdy Z$ slides into meridianal form on this truncation face and an adjacent one. The resulting surface is still essential and meridianally incompressible, and still in meridianal form, so it can be isotoped further into normal form in a manner that does not increase intersections with interior edges, using \refthm{Normal}.
This contradicts the minimum weight assumption.
 
Next suppose $\alpha$ has both endpoints on truncation edges, as in~(e) or~(f) of \reffig{Cases2Saddles}. (Note that in the classical setting, this case is automatically excluded by the definition of standard position~\cite{MenascoThistlethwaite}.) Away from $\beta$, push $D$ slightly into the interior of the chunk $C$, similar to the move on the left of \reffig{NoSaddleTwice}. Then the newly isotoped disk has the form of a boundary compression disk $D'$ for $Z$. 
Since $Z$ is boundary incompressible, we may push off this disk using a standard argument (see for example~\cite[Figure~1.9]{Hatcher:3MfldNotes}), as follows. 
By boundary incompressibility of $Z$, $\bdy D' \cap Z$ must cobound a disk $E\subset Z$ with an arc $\delta$ of $Z\cap \bdy N(L)$. Because $Y-N(L)$ is boundary irreducible by \cite[Corollary~3.16]{HowiePurcell}, $D'\cup E$ must be parallel to a disk $F\subset \bdy N(L)$. Then $D'\cup E\cup F$ is a sphere in $Y-N(L)$ bounding a ball; use the ball to isotope $Z$, moving the disk $E$ through the ball and slightly past $D'$. After this isotopy, $Z$ has at least two fewer intersections with truncation edges, and no aditional intersections with other interior edges, so the weight has strictly decreased. Observe also that the arc $\delta\subset \bdy Z$ could not have been in meridianal form, so this case will not affect meridianal boundary components of $Z$, and hence $Z$ remains in meridianal form after isotopy.  
Then further isotopy as in \refthm{Normal} puts $Z$ back into normal form but with strictly reduced weight, contradicting our assumption that $Z$ had minimal weight. Hence we may assume that an innermost instance of $D$ does not appear as in~(e) or~(f) of \reffig{Cases2Saddles}.

Finally suppose an innermost $D$ has the form of either~(c) or~(d) in \reffig{Cases2Saddles}, with one endpoint of $\alpha$ on an interior edge and one on a truncation edge. We will isotope $Z$ to remove the intersection with the interior edge in two steps. The first step is to slide $\bdy Z$ through $\bdy N(L)$ in $Y-N(L)$. For~(c), we slide through a neighborhood of $\beta\cap \bdy N(L)$ as shown on the left of \reffig{NoSaddleAdjTruncFace}. For~(d), note that because $D$ is innermost, the arc of $Z\cap \bdy C$ that meets the truncation face cannot exit the truncation face at the edge vertically adjacent (so the arc cannot be in meridianal form, meaning the corresponding component of $\bdy Z$ is not a meridian). 
So the arc on the truncation face continues to exit through one of the two other edges on the truncation face. One of these is shown in \reffig{NoSaddleAdjTruncFace} (d). In both cases, we slide $\bdy Z$ through the truncation face, past the point where it meets the interior edge, as shown on the far right in the figure. The effect of these moves on the component of $Z\cap \bdy C$, on the disk $D$, and on $\bdy Z$ in $\bdy N(L)$ is shown.
Observe that this first step of the isotopy temporarily increases weight by one.

\begin{figure}[h]
\begingroup%
  \makeatletter%
  \providecommand\color[2][]{%
    \errmessage{(Inkscape) Color is used for the text in Inkscape, but the package 'color.sty' is not loaded}%
    \renewcommand\color[2][]{}%
  }%
  \providecommand\transparent[1]{%
    \errmessage{(Inkscape) Transparency is used (non-zero) for the text in Inkscape, but the package 'transparent.sty' is not loaded}%
    \renewcommand\transparent[1]{}%
  }%
  \providecommand\rotatebox[2]{#2}%
  \newcommand*\fsize{\dimexpr\f@size pt\relax}%
  \newcommand*\lineheight[1]{\fontsize{\fsize}{#1\fsize}\selectfont}%
  \ifx\svgwidth\undefined%
    \setlength{\unitlength}{332.67404938bp}%
    \ifx\svgscale\undefined%
      \relax%
    \else%
      \setlength{\unitlength}{\unitlength * \real{\svgscale}}%
    \fi%
  \else%
    \setlength{\unitlength}{\svgwidth}%
  \fi%
  \global\let\svgwidth\undefined%
  \global\let\svgscale\undefined%
  \makeatother%
  \begin{picture}(1,0.34337188)%
    \lineheight{1}%
    \setlength\tabcolsep{0pt}%
    \put(0,0){\includegraphics[width=\unitlength,page=1]{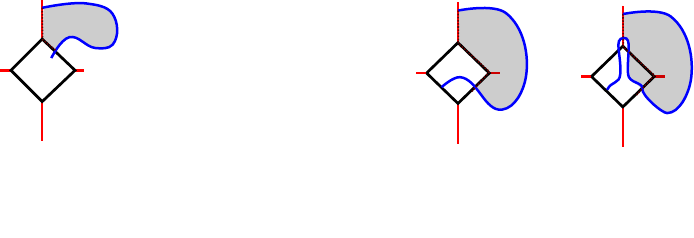}}%
    \put(0.00438847,0.14577964){\color[rgb]{0,0,0}\makebox(0,0)[lt]{\lineheight{1.25}\smash{\begin{tabular}[t]{l}(c)\end{tabular}}}}%
    \put(0.60170873,0.13939763){\color[rgb]{0,0,0}\makebox(0,0)[lt]{\lineheight{1.25}\smash{\begin{tabular}[t]{l}(d)\end{tabular}}}}%
    \put(0,0){\includegraphics[width=\unitlength,page=2]{SlideSaddleAdjTruncFace.pdf}}%
  \end{picture}%
\endgroup%

\caption{Isotope $Z$ by sliding $\bdy Z$ in a neighborhood of $\beta \cap \bdy N(L)$. Effect is shown on the component in the chunk on top, and on $\bdy N(L)$ on bottom.}
\label{Fig:NoSaddleAdjTruncFace}
\end{figure}

The next step is identical to the process in case~(a) in \reffig{Cases2Saddles}: there is now a disk $D''$ in $\bdy C$ with boundary consisting of an arc on an interior edge and an arc of $Z\cap \bdy C$. Away from the interior edge, push this disk into $C$. Use it to isotope $Z$ via a band move as in \reffig{NoSaddleTwice}. This will remove all intersections of $Z$ with interior edges meeting $D''$. Thus it removes two intersections, and the weight strictly decreases from its initial value. If the corresponding component of $\bdy Z$ was not a meridian in meridianal form,  this gives a contradiction to the fact that it was least weight: further isotopy as in \refthm{Normal} will put $Z$ back into normal form without increasing weight.

If the corresponding component of $\bdy Z$ is a meridian, and so it began in meridianal form, we may need to isotope further to put it back into meridianal form to apply \refthm{Normal}. As noted when we introduced case~(d), if this component of $\bdy Z$ is a meridian then a disk as in case~(d) in \reffig{Cases2Saddles} cannot be innermost. Thus we are in case~(c). Then we follow the proof of \refthm{Normal}: the initial isotopy of $\bdy Z$ will adjust a curve in meridianal form exactly as in \reffig{Condition4-BdryView}, 
yielding an arc on a truncation face with both endpoints on the same truncation edge. Just as in the proof of \refthm{Normal}, further isotopy will put the surface $Z$ back into meridianal form exactly as shown in Figures~\ref{Fig:NormalMeridIsotope1} and~\ref{Fig:NormalMeridIsotope2}. This gives a surface that is essential and in meridianal form with smaller weight, and so further isotopy using \refthm{Normal} will give a normal surface with smaller weight; a contradiction. 
\end{proof}

\section{Angled chunks and combinatorial area}\label{Sec:Area}

So far, we have isotoped essential surfaces into normal form, and have considered some of the resulting subsurfaces in chunks and their boundary curves. In this section, we recall an important tool to reduce the options for such subsurfaces, namely combinatorial area.

In its most general form, combinatorial area can be defined for surfaces in any chunk with dihedral angles assigned to interior edges that satisfy certain conditions; this is called an \emph{angled chunk decomposition}, and it is described in full generality in \cite{HowiePurcell}. In our setting, label each interior edge with angle $\pi/2$. By \cite[Proposition~3.15]{HowiePurcell}, the decomposition satisfies the requirements to be an angled chunk decomposition.

We now review a few consequences.

\subsection{Combinatorial area}

Let $Z$ be in normal form with respect to the chunk decomposition of $Y-N(L)$. Write $Z=\bigcup_{j=1}^m Z_j$, where each $Z_j$ is a connected normal surface embedded in a chunk.
Consider $\bdy Z_j$. Each component of $\bdy Z_j$ meets interior edges a total of $n_S = n_S(Z_j)$ times;
this is the number of instances of $S$ on the words decorating $\bdy Z_j$. It meets truncation edges a total of $n_T=n_T(Z_j)$ times, where $n_T$ must be even, since $Z_j$ enters and exits each truncation face by meeting a truncation edge. In case $Z$ is meridianal, $n_T/2$ is the total number of instances of $P$ in the words decorating $\bdy Z_j$. Otherwise, $n_T$ is the total number of instances of $B$.

The \emph{combinatorial area} of $Z_j$ is defined to be
\begin{equation}\label{Eqn:SubArea}
  a(Z_j) = \frac{\pi}{2}n_S + \frac{\pi}{2}n_T - 2\pi\chi(Z_j).
\end{equation}
Denote the number of $S$'s in $Z_j$ by $\#S$, and the number of $P$'s by $\#P$. If $Z$ is meridianal, we rewrite \refeqn{SubArea} as
\begin{equation}\label{Eqn:SubAreaMeridianal}
  a(Z_j) = \frac{\pi}{2}(\# S) + \pi(\# P) - 2\pi\chi(Z_j).
\end{equation}
The \emph{combinatorial area} of $Z$ is defined to be
\begin{equation}
  a(Z) = \sum_{j=1}^m a(Z_j).
\end{equation}

The combinatorial area satisfies a Gauss--Bonnet formula~\cite[Proposition~3.12]{HowiePurcell}:
\begin{equation}\label{Eqn:GB}
a(Z) = -2\pi\chi(Z).
\end{equation}

We also have the following results that follow from~\cite{HowiePurcell}.

\begin{lemma}\label{Lem:Areas}
Let $Z_j$ be a subsurface of a connected normal surface $Z$. The combinatorial area of $Z_j$ satisfies the following.
\begin{enumerate}
\item[(1)] If $\chi(Z_j)< 0$, then $a(Z_j)\geq 2\pi$.
\item[(2)] If $\chi(Z_j)\geq 0$, then either $a(Z_j)\geq \pi/2$, or $a(Z_j)=0$.
\item[(3)] Additionally, in the case when $a(Z_j)=0$, $Z_j$ is either:
  \begin{enumerate}
  \item[(a)] an essential torus or Klein bottle embedded in a chunk, hence $Z=Z_j$ is a torus or Klein bottle,
  \item[(b)] an annulus or M\"obius band with boundary meeting no edges, or
  \item[(c)] a disk such that $\bdy Z_j$ meets exactly four edges of the chunk decomposition.
  \end{enumerate}
\end{enumerate}
\end{lemma}

\begin{proof}
Equation (\ref{Eqn:SubArea}) in the definition of combinatorial area implies the first item. 

Note that since $Z_j$ is a connected surface, with or without boundary, and $Z_j$ is never a sphere or projective plane by~\refprop{TruncationEdgesPairs}~(2), we have $\chi(Z_j) \leq 1$.

If $\chi(Z_j)=0$, then $Z_j$ is either an annulus, M\"obius band, torus, or Klein bottle. If $Z_j$ is an annulus or M\"obius band and $\bdy Z_j$ meets an edge, then $a(Z_j)\geq \pi/2$ by equation~\refeqn{SubArea}. If $Z_j$ is a torus or Klein bottle, it lies in the interior of the chunk, so meets no edges, and $a(Z_j)=0$. Similarly if it is an annulus or M\"obius band such that $\bdy Z_j$ meets no edges, then $a(Z_j)=0$.

If $\chi(Z_j)=1$, then $Z_j$ is a disk. Then $a(Z_j) = k\pi/2 -2\pi$ by formula \refeqn{SubArea}, where $k$ is the number of edges (interior or truncation) met by $\bdy Z_j$.  Thus if $k>4$, then $a(Z_j)\geq \pi/2$. If $k=4$, then $a(Z_j)=0$ and~(3)(c) holds. Finally, the cases $k=0,1,2,3$ are ruled out by~\cite[Proposition~3.11]{HowiePurcell}.
\end{proof}

The Gauss--Bonnet formula, equation~\refeqn{GB}, implies that there are no normal spheres or normal disks embedded in a weakly generalized alternating link complement, for such a surface must have negative combinatorial area, and hence it must have a subsurface of negative combinatorial area, which is impossible by \reflem{Areas}.
Howie and Purcell observe that if a weakly generalized alternating link is irreducible, it must contain a normal sphere, and if it is boundary irreducible, it must contain a normal disk~\cite[Theorem~3.8]{HowiePurcell}.\footnote{This does not follow from our proof of \refthm{Normal}, because we assumed this result to prove \refthm{Normal}. However, the proof is similar, using surgery on disks rather than isotopy as in \cite{fg09}.
Hence a corollary is that weakly generalized alternating link complements are irreducible and boundary irreducible~\cite[Corollary~3.16]{HowiePurcell}.} Additionally, that paper studies surfaces with combinatorial area zero, namely annuli and tori, to analyze when weakly generalized alternating knots are hyperbolic~\cite[Section~4]{HowiePurcell}. By contrast, the work in this paper allows us to study more surfaces: higher genus surfaces, punctured spheres, etc. For a surface with fixed Euler characteristic, the Gauss--Bonnet formula \refeqn{GB} restricts potential normal subsurfaces to those with combinatorial area no more than the original. This paper helps us analyze the structure of potential normal subsurfaces, and how they interact with the diagram. We will see that of key importance are subsurfaces with zero combinatorial area.

\begin{lemma}\label{Lem:0AreaEnumerated}
In a weakly generalized alternating link complement,
the only normal subsurfaces that are disks with zero combinatorial area have boundaries labeled $PP$, $PSS$, $SSSS$, $BBBB$, and $BBSS$.
\end{lemma}

Recall that we do not allow both $P$ and $B$ labels in the same word. Thus we will regard $PBB$ disks as instances of $BBBB$ disks, as in \refrem{BBReplacesP}. 

\begin{proof}[Proof of \reflem{0AreaEnumerated}]
By \reflem{Areas}, a zero area disk meets four edges. If these are all interior edges, the disk is labeled $SSSS$. If it meets truncation edges, it must do so in pairs, thus either two adjacent instances of $B$ or a single $P$ (encoding two truncation edges for a curve in meridianal form). Thus the possibilities are as claimed.
\end{proof}

\section{Eliminating $SSSS$ disks}\label{Sec:SSSS}

In this section, we restrict disks of the form $SSSS$.

\begin{theorem}\label{Thm:NoSSSS}
  Suppose $Z$ is meridianally incompressible surface in normal form with respect to the chunk decomposition of $Y-N(L)$. Then any $SSSS$ disk is an essential compression disk for $\Pi$ meeting the diagram $\pi(L)$ exactly four times. Thus if the representativity satisfies $r(\pi(L),\Pi)>4$, there are no $SSSS$ disks.
\end{theorem}

\begin{proof}
Suppose $Z_i$ is a normal disk meeting exactly four interior edges and suppose $Z_i$ is not a compressing disk for $\Pi$. Then $Z_i$ is parallel into the boundary surface $\Pi^+$ or $\Pi^-$ of a chunk, without loss of generality say $\Pi^+$, so the curve $\bdy Z_i$ bounds a disk on $\Pi^+$ (meeting edges, faces, etc. in its interior).

\begin{figure}[h]
  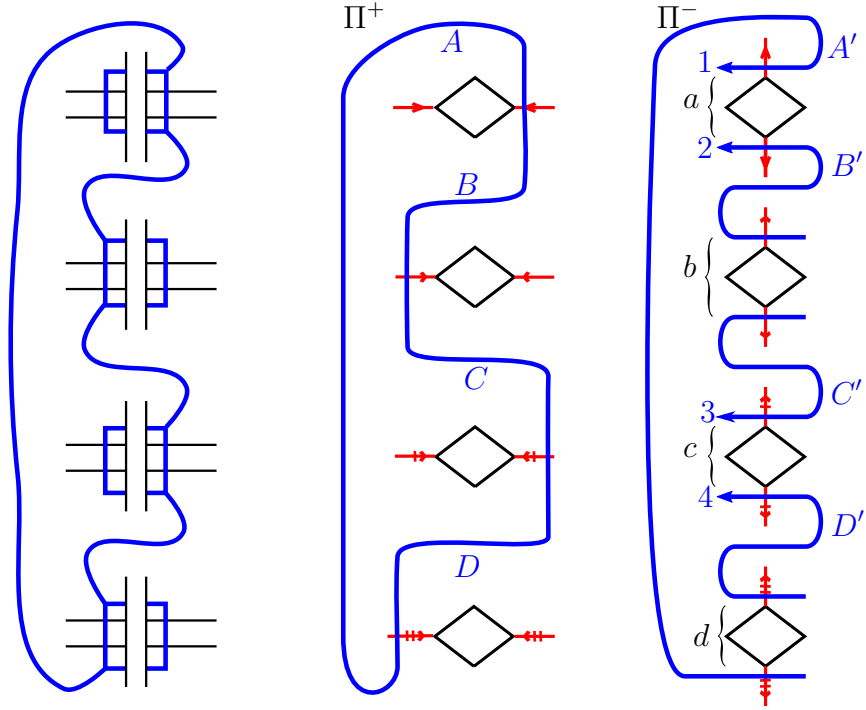
  \caption{Left: Disk with boundary SSSS in $Y-N(L)$. Middle: form in chunk decomposition with boundary $\Pi^+$. Right: glued to $\Pi^-$ as shown.}
  \label{Fig:SSSS}
\end{figure}

The form of $Z_i$ in the diagram $\pi(L)$ is shown on the left of \reffig{SSSS}, where the curve $\bdy Z_i \cap \Pi$ is shown in blue along with four saddles. In particular, if one follows the component of $\bdy Z_i\cap \Pi$, the overpasses of $\pi(L)$ alternate between being on the right and on the left due to the diagram $\pi(L)$ being alternating. The form of $\bdy Z_i$ in one chunk is shown in the middle, where $\bdy Z_i$ lies on some component $\Pi_j^+$ of $\Pi^+$. Denote the four arcs of $\bdy Z_i$ between the intersection points with edges by $A, B, C, D$ as on the middle figure. The surface $\Pi_j^+$ is glued to $\Pi_j^-$ via a gluing that rotates the boundary of each face as in subsection~\ref{Sec:Gluing}. Thus the arcs $A$, $B$, $C$, $D$ are glued to four arcs $A'$, $B'$, $C'$, $D'$ as shown on the right. The blue arcs in the right figure are arcs of boundary curves of other normal subsurfaces $Z_k$ in the chunk decomposition; thus they are portions of simultaneously embedded closed curves.

First, observe that while $A'$, $B'$, $C'$, and $D'$ no longer necessarily connect to bound a disk in $\Pi^-$, the union of the arcs $A'$, $B'$, $C'$, and $D'$ along with meridianal arcs through the four truncation faces shown on the right of \reffig{SSSS}, and pieces of interior edges between them, all bound a disk on $\Pi_j^-$. 

Label the blue arcs that run into this disk in $\Pi^-$ by 1, 2, 3, 4 as in \reffig{SSSS}, right. Because the arcs are subsets of a disjoint union of embedded closed curves (the boundaries of normal surfaces in the chunk), the arcs $1$, $2$, $3$, and $4$ must exit the disk region shown. They might exit either by connecting to each other, for example $1$ might connect to $2$, or they might exit by running through the region between $1$ and $2$, between arcs $B'$ and $C'$, between $3$ and $4$, or between arcs $D'$ and $A'$. We call these regions zones $a$, $b$, $c$, and $d$, and they are shown in \reffig{SSSS}, right, as well.

The arcs labeled 1 and 2 cannot connect to each other or exit through zone $a$ by \refprop{NoSaddleTwice}, since they each meet the interior edge in that region once already. Similarly, the arcs 3 and 4 cannot connect with each other or exit through zone $c$. Note also that arcs 1 and 4 cannot run to zone $d$ by the same result, since arc $A'$ already meets the edge of this region, and similarly for $D'$. Similarly arcs 2 and 3 cannot run through zone $b$.

It follows that the arc labeled $1$ runs into zone $b$ or $c$. If $b$, then the arc labeled $2$ must run into zone $a$ or $b$ in order to remain disjoint from the arc labeled $1$. But this is impossible: these are exactly the zones that $2$ is not allowed to enter. Hence the arc labeled $1$ runs to zone $c$. But then the arc labeled $4$ must run into zone $c$ or $d$ in order to remain disjoint from $1$. Again this is impossible: these are exactly the zones that $4$ is not allowed to enter. It follows that there can be no $SSSS$ disk that is not a compressing disk for $\Pi$.
\end{proof}

\section{Eliminating disks with words in the letter $P$}
\label{Sec:PPPP}

In this section, we prove that if $Z$ has meridianal boundary, there are no disks whose boundaries are curves of intersection labeled with two or three letters in $P$ and $S$. That is, there are no words of type $SS$, $PP$, $PS$, $PPP$, $PSS$, $SSS$, or $PPS$. The proof for $SS$ and $SSS$ holds in more generality, and does not require meridianal boundary.
We also prove that for any subsurface $Z_i$ of $Z$, there are no boundary components $\bdy Z_i$ labeled with an odd number of instances of $P$ and $S$. Finally, we use this to show all weakly generalized alternating links are prime.

\begin{theorem}\label{Thm:No2LetterWords}
Suppose $Z$ is in normal form with respect to the chunk decomposition of $Y-N(L)$. Then no normal subsurface of $Z$ is a disk whose boundary is labelled $SS$, nor a disk in meridianal form whose boundary is labelled $PP$ or $PS$. 
\end{theorem}

\begin{proof}
Recall that the boundary of a chunk is decorated by the diagram graph, where interior edges corresond to diagram edges, and truncation faces correspond to neighborhoods of crossings. Therefore the boundary of a disk labeled by two letters $PP$, $PS$, or $SS$ gives a curve on $\Pi$ meeting the diagram $\pi(L)$ at edges or at crossings.

In the case of a label $P$, there is a corresponding arc through a truncation face in meridianal form. We isotope this curve very slightly off the crossing in a direction determined by the side of the truncation face that is met by the curve. See \reffig{IsotopeP}, where an example is shown of a boundary $PP$. The isotopy moves $\bdy Z_i$ slightly to the curve $\gamma$, which we think of as lying on $\Pi$ and meeting two edges of the diagram graph $\pi(L)$.

\begin{figure}
\begingroup%
  \makeatletter%
  \providecommand\color[2][]{%
    \errmessage{(Inkscape) Color is used for the text in Inkscape, but the package 'color.sty' is not loaded}%
    \renewcommand\color[2][]{}%
  }%
  \providecommand\transparent[1]{%
    \errmessage{(Inkscape) Transparency is used (non-zero) for the text in Inkscape, but the package 'transparent.sty' is not loaded}%
    \renewcommand\transparent[1]{}%
  }%
  \providecommand\rotatebox[2]{#2}%
  \newcommand*\fsize{\dimexpr\f@size pt\relax}%
  \newcommand*\lineheight[1]{\fontsize{\fsize}{#1\fsize}\selectfont}%
  \ifx\svgwidth\undefined%
    \setlength{\unitlength}{138.75001144bp}%
    \ifx\svgscale\undefined%
      \relax%
    \else%
      \setlength{\unitlength}{\unitlength * \real{\svgscale}}%
    \fi%
  \else%
    \setlength{\unitlength}{\svgwidth}%
  \fi%
  \global\let\svgwidth\undefined%
  \global\let\svgscale\undefined%
  \makeatother%
  \begin{picture}(1,0.49048786)%
    \lineheight{1}%
    \setlength\tabcolsep{0pt}%
    \put(0,0){\includegraphics[width=\unitlength,page=1]{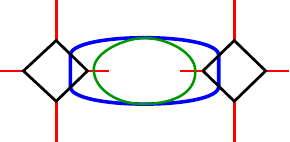}}%
    \put(0.37171808,0.18530218){\color[rgb]{0,0.58823529,0}\makebox(0,0)[lt]{\lineheight{1.25}\smash{\begin{tabular}[t]{l}$\gamma$\end{tabular}}}}%
    \put(0.66216214,0.3668661){\color[rgb]{0,0,1}\makebox(0,0)[lt]{\lineheight{1.25}\smash{\begin{tabular}[t]{l}$\bdy Z_i$\end{tabular}}}}%
  \end{picture}%
\endgroup%

  \caption{If $\bdy Z_i$ is labeled $PP$, it determines a curve $\gamma$ meeting the diagram exactly twice.}
  \label{Fig:IsotopeP}
\end{figure}

Then in all cases, the disk $Z_i$ determines a curve $\gamma$ on $\Pi$ meeting the diagram exactly twice transversely in diagram edges. Because $Z_i$ is a disk, $\gamma$ must also bound a disk. Because the diagram has representativity at least $4$, by definition of weakly generalized alternating, the disk must be parallel to $\Pi$. Then the fact that the diagram is weakly prime implies that $\gamma$ bounds a disk that does not meet any crossings.

In the case that the curve is of the form $SS$ or $PS$, this gives an immediate contradiction: the curve $\bdy Z_i$ is not in normal form, violating condition~(3) or~(4) of \refdef{NormalSurface}, respectively.

In the case the curve is of the form $PP$, then it has the form shown in \reffig{IsotopeP}. Both arcs are in meridianal form, thus they cut off a corner that is not identified to another corner of the truncation face. However, consider the interior edge encircled by the curve $\bdy Z_i$. 
Because this is an interior edge of a chunk decomposition of a weakly generalized alternating link,
at one of its endpoints, this edge is identified to another edge across the corresponding truncation face, as described in \refsec{Gluing}; see also \reffig{CrossingArc}. Thus one of its endpoints meets a corner of a truncation face that is identified to another corner of the same truncation face. This is a contradiction: because $\bdy Z_i$ is in meridianal form, neither endpoint can have this property.
\end{proof}

Note that the previous proof relies heavily on the alternating condition, when using the chunk decomposition of \refsec{Gluing}.

\begin{theorem}\label{Thm:No3LetterWords}
Suppose $Z$ is in normal form with respect to the chunk decomposition of $Y-N(L)$. Then no normal subsurface of $Z$ is a disk whose boundary is labelled $SSS$, nor a disk in meridianal form whose boundary is labelled $PPP$, $PPS$, or $PSS$. 

More generally, no normal subsurface $Z_i$ of $Z$ has a boundary component meeting an odd number of letters $S$ and $P$.
\end{theorem}

\begin{proof}
As in the proof of \refthm{No2LetterWords}, a curve $\bdy Z_i$ determines an embedded curve $\gamma$ on the diagram graph $\pi(L) \subset \Pi$, with letters $S$ running transversely through diagram edges, and letters $P$ running through a diagram edge immediately to the left or right of a crossing, determined by the position of the arc on the truncation face in meridianal form as in \reffig{IsotopeP}.

In each case, consider the checkerboard coloring of the diagram. Passing through a letter $P$ or $S$ changes the color of the face meeting $\bdy Z_i$. If the word labelling $\bdy Z_i$ is made up of exactly three letters, then the color must change exactly three times. But this is impossible: if we start in a white face and traverse $\bdy Z_i$, it changes to shaded when it meets the first letter, to white when it meets the second letter, to shaded when it meets the third, and then it must close up, implying that the starting and ending face is both shaded and white. This contradiction is illustrated in the case $PSS$ in \reffig{No3LetterWords}.
\begin{figure}[h]
\begingroup%
  \makeatletter%
  \providecommand\color[2][]{%
    \errmessage{(Inkscape) Color is used for the text in Inkscape, but the package 'color.sty' is not loaded}%
    \renewcommand\color[2][]{}%
  }%
  \providecommand\transparent[1]{%
    \errmessage{(Inkscape) Transparency is used (non-zero) for the text in Inkscape, but the package 'transparent.sty' is not loaded}%
    \renewcommand\transparent[1]{}%
  }%
  \providecommand\rotatebox[2]{#2}%
  \newcommand*\fsize{\dimexpr\f@size pt\relax}%
  \newcommand*\lineheight[1]{\fontsize{\fsize}{#1\fsize}\selectfont}%
  \ifx\svgwidth\undefined%
    \setlength{\unitlength}{125.3113203bp}%
    \ifx\svgscale\undefined%
      \relax%
    \else%
      \setlength{\unitlength}{\unitlength * \real{\svgscale}}%
    \fi%
  \else%
    \setlength{\unitlength}{\svgwidth}%
  \fi%
  \global\let\svgwidth\undefined%
  \global\let\svgscale\undefined%
  \makeatother%
  \begin{picture}(1,0.54308828)%
    \lineheight{1}%
    \setlength\tabcolsep{0pt}%
    \put(0,0){\includegraphics[width=\unitlength,page=1]{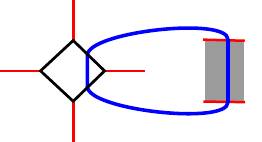}}%
    \put(0.39884658,0.18176869){\color[rgb]{0,0,1}\makebox(0,0)[lt]{\lineheight{1.25}\smash{\begin{tabular}[t]{l}$P$\end{tabular}}}}%
    \put(0.8797916,0.4254479){\color[rgb]{0,0,1}\makebox(0,0)[lt]{\lineheight{1.25}\smash{\begin{tabular}[t]{l}$S$\end{tabular}}}}%
    \put(0.87124148,0.05779251){\color[rgb]{0,0,1}\makebox(0,0)[lt]{\lineheight{1.25}\smash{\begin{tabular}[t]{l}$S$\end{tabular}}}}%
  \end{picture}%
\endgroup%

  \caption{A $PSS$ disk should have three arcs in three faces of distinct colors, but this is impossible for a checkerboard colored diagram.}
  \label{Fig:No3LetterWords}
\end{figure}
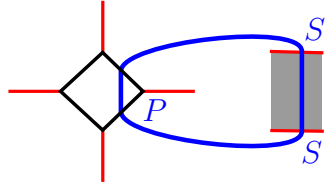

More generally, it is impossible for a curve of $\bdy Z_i$ to meet an odd number of letters $S$ and $P$, since again the existence of such a curve would contradict the checkerboard coloring of the diagram.
\end{proof}

The above theorem gives us a quick way to prove the fact that a  weakly generalized alternating link with a cellular diagram is prime. When the hat-representativity satisfies $\hat{r}(\pi(L), \Pi)>4$, this result was originally proved in Howie--Purcell~\cite[Corollary~4.7]{HowiePurcell}. We can extend now to all weakly generalized alternating links without the additional hat-representativity condition.

\begin{theorem}\label{Thm:Prime}
  A weakly generalized alternating link is prime.
\end{theorem}

\begin{proof}
Suppose not. Then there exists an essential meridianal annulus $Z$. Put it into normal form with respect to the chunk decomposition. We may assume it decomposes into normal subsurfaces that are meridianal by \refthm{Normal}, and have zero combinatorial area by the Gauss--Bonnet formula,  equation~\refeqn{GB}. 
Because it meets two meridians of the link, at least one of the subsurfaces $Z_i$ must have a boundary curve $\bdy Z_i$ whose labelling includes at least one instance of $P$. Because both boundary components are meridianal, there will be no instances of $B$. By \reflem{Areas}, $Z_i$ must be a disk meeting exactly four edges of the chunk decomposition. By \reflem{0AreaEnumerated}, the possibilities are disks with boundaries labeled $PP$ or $PSS$. Disks labeled $PP$ are ruled out by \refthm{No2LetterWords}. Disks labeled $PSS$ are ruled out by \refthm{No3LetterWords}. This gives a contradiction.
\end{proof}

\section{Disks with $BBBB$ and $BBSS$ words}\label{Sec:BBBB}

This section concerns disks whose boundaries are $BBBB$ or $BBSS$ words. As mentioned in \refrem{BBReplacesP}, an instance of $P$ may be replaced with two instances of $B$, and the results for $BB$ go through. Hence the results for $BBBB$ disks immediately apply to $PBB$ disks, and we will use this in the sequel.

For a surface $Z$ with non-meridianal boundary, consider a normal disk $Z_i$ that meets exactly four truncation edges, and no interior edges. Such a subsurface is a disk that corresponds to a $BBBB$ word. For an example, see \reffig{BBBBExample}, left. Similarly, a normal disk that meets exactly two truncation edges and exactly two interior edges corresponds to a $BBSS$ word. An example is shown in \reffig{BBBBExample}, right.

\begin{figure}[h]
  \includegraphics{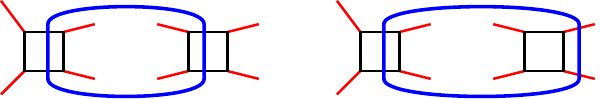}
  \caption{Left: an example of a normal disk of type $BBBB$. Right: an example of type $BBSS$.}
  \label{Fig:BBBBExample}
\end{figure}

As described in \cite{HTT2}, topologically the label $B$ means that $\partial Z_i$ meets $\bdy N(L)$. An arc $BB$ can be a part of $\partial Z$, where $\partial Z$ is on $\bdy N(L)$. In a chunk decomposition, such an arc travels between two truncation edges in a truncation face. An arc $BB$ can also connect two $BB$ arcs of the previous type, and then it lies in an interior face of a chunk. Each $BBBB$ disk is hence a quadrilateral with two opposite sides on truncation faces and the other pair of opposite sides on interior faces. Since a $BB$-arc that lies in a truncation face is a part of $\partial Z$, it is not glued to any other arc of $\partial Z_j$, for any $j$. On the other hand, a $BB$-arc that lies inside an interior face, as well as a $BS$-arc, is not a part of $\partial Z$, and is rather in the interior of $Z$. Such an arc therefore must be glued to a similar arc of some $Z_k$.

We say that two $BBBB$ disks are \emph{connected} if they share a common arc on an interior face. A collection of such disks is \emph{connected} if the union of disks is connected in $Y$. A collection of connected disks is \emph{maximal} if it is not a strict subset of a connected collection. 
We define connected and maximal collections of $BBSS$ disks similarly. 

\begin{lemma}\label{Lem:BBSSnoSS}
Suppose $D$ is a $BBSS$ disk that is parallel to the surface $\Pi$. Then $D$ is not connected to another $BBSS$ disk along its $SS$ arc.
\end{lemma}

\begin{proof}
A $BBSS$ disk meets a truncation face along the arc $BB$, two opposite faces adjacent to that truncation face of the same color --- say white --- along the two arcs $BS$, and a face containing the $SS$ arc that is shaded. 
Without loss of generality, say the $BBSS$ disk parallel to $\Pi$ lies on $\Pi^+$. Then the three arcs of the disk that lie in interior faces are glued to arcs in $\Pi^-$. As in \refsec{Gluing}, the gluing of chunks in the neighborhood of edges is via a rotation of the respective faces: in the clockwise direction for the two arcs $BS$, and in a counterclockwise direction for the arc $SS$. Superimpose all these arcs on $\Pi$ as in \reffig{BBSSnoSS}, left. The boundary of the disk $BBSS$ is shown in red, and the arcs in $\Pi^-$ that it is glued to are shown in light blue. 

\begin{figure}
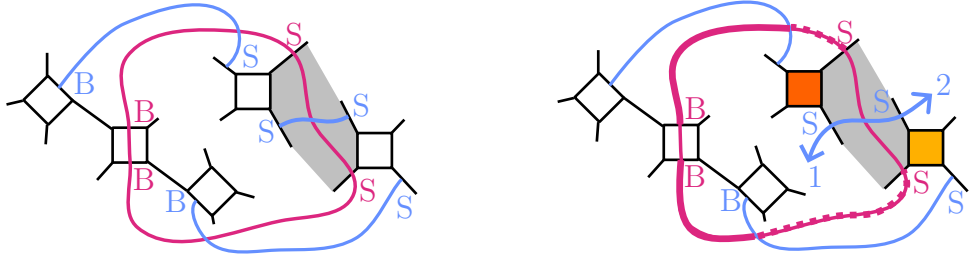
\caption{Left: $\bdy D$ is shown in red in $\Pi^+$. Arcs of subsurfaces meeting $\Pi^-$ that are glued to $\bdy D$ are superimposed on $\Pi^+$, and shown in light blue. Right: If the blue $SS$ arc connects to form the boundary of a $BBSS$ disk, it cannot meet thick or dashed red arcs.}
\label{Fig:BBSSnoSS}
\end{figure}

By way of contradiction, suppose the $SS$ arc glues to an $SS$ arc contained in another $BBSS$ disk $E$, with boundary on $\Pi^-$.  By hypothesis, $\bdy E$ is labelled with only two instances of $S$, hence $\bdy E$ cannot contain either of the two blue arcs $BS$ in white faces obtained by rotations of arcs of $D$. Thus the disk $E$ must be disjoint from these arcs. Also note that (the superimposed copy of) $\bdy E$ intersects $\bdy D$. Because $D$ is a disk parallel into $\Pi$, $\bdy E$ must meet $\bdy D$ one additional time.

Note that the arc of $\bdy E$ labelled $1$ cannot run immediately into the dark orange truncation face; this would contradict the fact that $E$ is normal. Similarly the arc labelled $2$ cannot 
run immediately into the light orange truncation face, or $E$ is not normal. This means $\bdy E$ does not meet $\bdy D$ on the dashed arcs shown on the right of \reffig{BBSSnoSS}. 

Suppose the arc labelled $1$  runs to meet one of the thicker red arcs shown. Then there are two cases to consider.

\textit{Case 1.} Suppose arc~1 intersects the upper left red arc. Then arc~1 cannot intersect the blue arc, so it must run into a truncation face or interior edge after intersecting the thick red arc. Because arc~1 has already met two instances of $S$, it cannot meet an interior edge and remain the boundary of a $BBSS$ disk. Since meeting the truncation face uses up all the letters $BBSS$ in its boundary, this means that the arc labelled $1$ meets no edges between $1$ and the thick red arc. It follows that the two white faces on opposite sides of the dark orange truncation face are the same, since one white face contains arc~1 and the other white face contains the thick red arc.  But this would contradict the link being weakly prime. 

\textit{Case 2.} Suppose arc~1 meets the thick red arc in the lower left part of the figure. Similarly to case~1, arc~1 then must enter the region between that arc and the blue arc through a truncation face, using up all the letters $BBSS$. But then arc~1 must connect immediately to the arc labelled $2$, without meeting additional interior or truncation edges. This means that the white face containing the arc labelled $2$ and the white face containing the thick red arc in the lower left part of the figure are the same. But these are the two white faces adjacent to the light orange truncation face, contradicting the link being weakly prime. 

Therefore, arc 1 does not meet any of the two thicker red arcs. The only remaining place that the arc labelled $1$ could meet the boundary of the red disk would be on the truncation face between instances of $B$, and it must do so by entering and exiting opposite sides of that truncation face, not the sides meeting the red arc. But those opposite sides lie in shaded faces, whereas arc~1 is in a white face. This is also impossible. 
\end{proof}

\begin{lemma}\label{Lem:BBBB}
Assume that the hat-representativity $\hat{r}(\pi(L),\Pi)>4$, and that $\pi(L)$ is not a string of bigons on $\Pi$. Let $Z$ be in normal position, with a non-meridianal boundary component. Then the following form disks: 
\begin{enumerate}
\item a maximal connected collection of $BBBB$ disks for $Z$ or
\item a maximal connected collection of $BBSS$ disks for $Z$. 
\end{enumerate}
\end{lemma}

\begin{proof}
As explained above, the $BBBB$ or $BBSS$ disks are connected only along arcs on interior faces, since the arcs on truncation faces are a part of the surface boundary.

For (1), a maximal connected collection is therefore glued end to end along opposite interior faces. The only way it could not form a disk is if it forms an essential annulus or M\"obius band, made entirely of $BBBB$ disks. In the case of an annulus, it was shown in \cite[Theorem~4.6]{HowiePurcell} that $\pi(L)$ is a string of bigons on $\Pi$ (here we use the hypothesis $\hat{r}(\pi(L),\Pi)>4$). In fact the proof also implies the same result for the case of the M\"obius band, because it shows that a string of $BBBB$ disks connect to bound a chain of bigons in the diagram graph.

For (2), arcs in the interior faces have the form $BS$ or $SS$. By \reflem{BBSSnoSS}, $SS$ arcs cannot glue to other $SS$ arcs in the maximal collection. All $BBSS$ disks in the collection must be glued along arcs of the form $BS$. If the connected collection of $BBSS$ disks is not a disk, it will form a normal annulus made up of $BBSS$ disks, at least one of which is parallel into $\Pi$ because $\hat{r}(\pi(L),\Pi)>4$. But then by~\cite[Lemma~4.5]{HowiePurcell}, the diagram $\pi(L)$ is a string of bigons on $\Pi$, contradicting our hypotheses.
\end{proof}

\begin{proposition}\label{Prop:NoBBBBandBBSS}
For a surface $Z$ in normal position, no $BBBB$ region is connected to a $BBSS$ region.
\end{proposition}

\begin{proof}
Two normal subsurfaces $Z_i, Z_j$ in a chunk can only be connected across arcs in interior faces, not across arcs in truncation faces, which are left unglued and become the boundary of the normal surface. The arcs in the interior faces for a $BBBB$ region have both endpoints on a truncation edge. The arcs in the interior faces for a $BBSS$ region have one end on a truncation edge and one end on an interior edge, or both ends on interior edges. Because interior edges glue to interior edges, each arc of a $BBSS$ region must glue to an arc with an endpoint on an interior edge. Thus it cannot glue to a $BBBB$ region.
\end{proof}

\begin{theorem}\label{Thm:NonBBBBdeterminesZ}
Assume that the hat-representativity $\hat{r}(\pi(L),\Pi)>4$, and $\pi(L)$ is not a string of bigons on $\Pi$. Let $Z$ be in normal form. Then all subsurfaces $Z_i$ that are neither $BBBB$ disks nor $BBSS$ disks, together with the link $L$, determine the surface $Z$ up to isotopy.
\end{theorem}

\begin{proof}
By \reflem{BBBB}~(1), maximal connected collections of $BBBB$ disks form disjoint disks. Similarly for $BBSS$ disks by \reflem{BBBB}~(2). By \refprop{NoBBBBandBBSS}, no such disks are connected. Therefore, if we remove maximal connected collections of $BBBB$ and $BBSS$ regions, we are left with a surface with a number of disjoint disks removed.

Suppose $Z'$ is another surface in normal form, whose normal subsurfaces agree with those of $Z$ away from $BBBB$ and $BBSS$ disks. Then $Z$ and $Z'$ are both obtained by adding disks to the same surface with boundary, along the same boundary components. The link $L$ has irreducible complement by~\cite[Theorem~3.14]{HowiePurcell}. Therefore, the two surfaces disagree except in disks that can be isotoped to bound balls. Thus the surfaces are isotopic.
\end{proof}


\section{Essential Conway spheres}\label{Sec:ConwaySpheres}

In this section, we apply our results above to generalize the work of Menasco~\cite{Menasco1984} that allows the characterization of essential 4-punctured spheres, or essential Conway spheres, for alternating links as ``visible'' and ``hidden'', with terminology due to Thistlethwaite~\cite{Thistlethwaite}. We call an essential Conway sphere \textit{visible} if its intersection with $\Pi^{\pm}$ consists of only one $PPPP$ curve, and \textit{hidden} if it consists of exactly two $PSPS$ curves.
See \reffig{VisibleHidden} and compare to~\cite[Figure~3(i)-(ii)]{Thistlethwaite}.

Note that if curves as in \reffig{VisibleHidden} appear in the classical alternating case, the region exterior to any dashed line will also be a disk, hence will contain an alternating tangle. In the weakly generalized alternating setting, this region may have higher genus. 

\begin{figure}
  \includegraphics{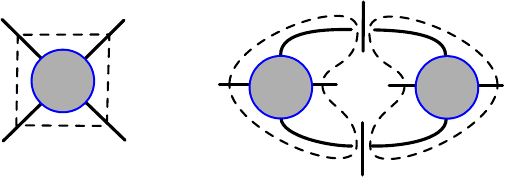}
  \caption{Left: a visible essential Conway sphere is made up of two disks labeled $PPPP$, one on either side of the projection surface. Right: a hidden Conway sphere is made up of four disks labeled $PSPS$, two on either side of the projection surface. Here we show how they meet one side of the diagram. The gray disk denotes a portion of the diagram contained in a disk in the projection surface. }
  \label{Fig:VisibleHidden}
\end{figure}

In this section, we will restrict to diagrams that are \emph{cellular}, meaning all regions of the diagram are disks. This is the only place in the paper where we use cellular.

\begin{theorem}\label{Thm:EssentialTangles}
Let $L$ be a weakly generalized alternating link on a connected projection surface $\Pi$ with a cellular diagram  and hat-representativity $\hat{r}(\pi(L),\Pi)>4$. Then any essential Conway sphere has one of two forms: visible or hidden.
\end{theorem}

We will prove the theorem by considering how essential Conway spheres decompose in a chunk decomposition, ruling out certain subsurfaces and labels on their boundaries. The proof will follow from a few lemmas.

\begin{lemma}\label{Lem:4PunctToDisks}
Under the hypotheses of \refthm{EssentialTangles}, suppose $Z=\bigcup Z_i$ is an essential Conway sphere in normal form with respect to the chunk decomposition. Then all the subsurfaces $Z_i$ must be disks.
\end{lemma}

\begin{proof}
An essential Conway sphere is an essential 4-punctured sphere, hence it will have combinatorial area $a(Z) = -2\pi\chi(S) = 4\pi$. Because it meets the knot in meridianal punctures, it will be in meridianal form. It decomposes into normal subsurfaces within chunks that are genus-zero surfaces. Thus the only subsurfaces that may arise are disks, annuli, and 3-holed spheres. Also, on each of $\Pi^{\pm}$ it must meet four instances of the letter $P$, so eight instances of $P$ total.

In a cellular link diagram, no boundary component of a genus-zero subsurface in normal position can lie entirely in a face of the diagram. This means that each boundary component of an annulus or 3-holed sphere must meet instances of $S$ or $P$, and by \refthm{No3LetterWords}, it must meet at least two such instances. We will prove any annulus or 3-holed sphere subsurface contributes too much area, to rule these out.

By equation \refeqn{SubAreaMeridianal}, the smallest possible area contribution from an annulus has area $2\pi$, meeting two instances of $S$ on each boundary component. But we still need eight instances of the letter $P$, and the total area of all subsurfaces must add up to $4\pi$. 
By \reflem{Areas}, all additional contributions to area are nonnegative.
Since each $P$ contributes area $\pi$ by equation \refeqn{SubAreaMeridianal}, annuli have too much area. Similarly for 3-holed spheres. Thus all subsurfaces must be disks.
\end{proof}

\begin{lemma}\label{Lem:4PunctDiskBdry}
Under the hypotheses of \refthm{EssentialTangles}, suppose $Z=\bigcup Z_i$ is an essential Conway sphere in normal form with respect to the chunk decomposition. Then each of the disks $Z_i$ must be labeled $SSSS$, $PSSS$, $PPSS$, $PSPS$, $PPPS$, or $PPPP$.
\end{lemma}

\begin{proof}
When $Z$ is an essential Conway sphere, by \reflem{4PunctToDisks}, all $Z_i$ are disks. Also, $Z$ must meet eight instances of $P$, with four instances each on $\Pi^{\pm}$. The total area is $a(Z)=4\pi$. Thus there is some $Z_i$ with positive area at most $2\pi$.

Consider the possible boundary words for disks with area at most $2\pi$.
Lemma~\ref{Lem:0AreaEnumerated}, along with the results of \refsec{PPPP}, rules out all zero area disks $PP$ and $PSS$. Any odd number of letters is ruled out by \refthm{No3LetterWords}. If the representativity of the diagram is 4, then 0-area $SSSS$ disks are possible, not ruled out by \refthm{NoSSSS}, although they will be ruled out for higher representativity.
The only disk contributing area $\pi/2$ has boundary $PSSS$. The possibilities contributing area $\pi$ are $PPSS$, $PSPS$, and $SSSSSS$. Those contributing area $3\pi/2$ are $PPPS$ and $PSSSSS$. Those contributing area $2\pi$ include $PPPP$, as well as other words with at most three instances of $P$ and some instances of $S$.

Each instance of $P$ in the disks enumerated above contributes at least $\pi/2$ to the sum. 
Returning to the case of an essential Conway sphere,
using the fact that there must be eight instances of $P$ in total, in fact each of the disks with positive area must contain a letter $P$, and the area of the disk will be $\pi/2$ times the number of instances of $P$. Because there must be four instances of $P$ on each side of $\Pi$, the only possibilities for disks overall are 0-area $SSSS$ disks, and positive area $PSSS$, $PPSS$, $PSPS$, $PPPS$, and $PPPP$ disks.
\end{proof}

By the assumption that $\hat{r}(\pi(L),\Pi)>4$ and the assumption that $\Pi$ is connected, the boundaries of essential compression disks on one side of $\Pi$ meet the diagram in more than four points. So on this side of $\Pi$, any $PSSS$, $PPSS$, or $SSSS$ curve bounds a disk parallel to the projection surface. Without loss of generality, say this side is $\Pi^+$. We note that unlike in the classical case of alternating links in $S^3$, here the proofs for $\Pi^+$ and $\Pi^-$ from this moment are not analogous.

We say that a $PPSS$ or $PSSS$ disk is \emph{innermost} on $\Pi^+$ if the boundary curve bounds a disk on the surface $\Pi^+$ that contains no other intersections with $Z$. We say it is \emph{outermost} if it bounds a disk $D$ on $\Pi^+$ such that $\Pi^+-D$ is a surface with boundary containing no intersections with $Z$.

\begin{lemma}\label{Lem:PSSSInnerOuter}
Under the hypotheses of \refthm{EssentialTangles}, if there exists a $PPSS$ disk or a $PSSS$ disk with boundary on $\Pi^+$, then there exists one that is innermost or one that is outermost with boundary on $\Pi^+$.
\end{lemma}

\begin{proof}
Suppose first that there is a $PPSS$ disk with boundary on $\Pi^+$. Because any disk with boundary meeting $\Pi^+$ at most four times is parallel to $\Pi^+$ (by the assumption on $\hat{r}$), we may assume the $PPSS$ curve bounds a disk in $\Pi^+$.
Since there are four instances of $P$ on each side of $\Pi$, there will either be another $PPSS$ disk, or a $PSPS$ disk, or two $PSSS$ disks with boundary on $\Pi^+$ by Lemma~\ref{Lem:4PunctDiskBdry}. 
In the first two cases, the second disk lies on one side of the $PPSS$ disk and so there are no curves $\bdy Z_i$ on the other side. In the last case, there may be a $PSSS$ disk on both sides, but then such a $PSSS$ disk is innermost. This proves the lemma when there is a $PPSS$ disk.

So now suppose there is a $PSSS$ disk with boundary on $\Pi^+$. Again, its boundary curve encloses a disk on $\Pi^+$.
Then within the same chunk there could be a single $PPPS$ disk, which means the $PSSS$ disk is innermost or outermost. There could be another $PSSS$ disk and a disk with two instances of $P$ (namely $PSPS$ or $PPSS$), in which case one of the two $PSSS$ disks will be innermost or outermost. Or there could be three additional $PSSS$ disks, in which case one is innermost or outermost.
\end{proof}

The following lemma can be viewed as an extension of a lemma of Menasco~\cite[Lemma~2]{Menasco1984} in our setting.

\begin{lemma}\label{Lem:NoPSSSorPPSS}
Under the hypotheses of \refthm{EssentialTangles}, suppose $Z=\bigcup Z_i$ is an essential Conway sphere in normal form with respect to the chunk decomposition. Then there are no $PSSS$, $PPSS$, $PPPS$, or $SSSS$ disks.
\end{lemma}

\begin{proof}
We will first rule out $PPSS$ and $PSSS$ disks with their boundaries in $\Pi^+$. Suppose there is such a disk. Then \reflem{PSSSInnerOuter} implies that there is either an innermost or outermost such disk.

The portion of the boundary in the $PPSS$ or $PSSS$ disk that runs between two instances of $S$ must lie in a single region of the link diagram; here we are using the fact that the edges and ideal vertices on $\Pi^+$ come from $\pi(L)$ as in Subsection~\ref{Sec:ChunkDecomp}. Because the diagram is alternating, it must meet the first instance of $S$ in a saddle with the link on one side, and in the second instance the link lies on the other side; see \reffig{NoPSSS} (a) and (b).

\begin{figure}
\begingroup%
  \makeatletter%
  \providecommand\color[2][]{%
    \errmessage{(Inkscape) Color is used for the text in Inkscape, but the package 'color.sty' is not loaded}%
    \renewcommand\color[2][]{}%
  }%
  \providecommand\transparent[1]{%
    \errmessage{(Inkscape) Transparency is used (non-zero) for the text in Inkscape, but the package 'transparent.sty' is not loaded}%
    \renewcommand\transparent[1]{}%
  }%
  \providecommand\rotatebox[2]{#2}%
  \newcommand*\fsize{\dimexpr\f@size pt\relax}%
  \newcommand*\lineheight[1]{\fontsize{\fsize}{#1\fsize}\selectfont}%
  \ifx\svgwidth\undefined%
    \setlength{\unitlength}{305.10420227bp}%
    \ifx\svgscale\undefined%
      \relax%
    \else%
      \setlength{\unitlength}{\unitlength * \real{\svgscale}}%
    \fi%
  \else%
    \setlength{\unitlength}{\svgwidth}%
  \fi%
  \global\let\svgwidth\undefined%
  \global\let\svgscale\undefined%
  \makeatother%
  \begin{picture}(1,0.3827596)%
    \lineheight{1}%
    \setlength\tabcolsep{0pt}%
    \put(0,0){\includegraphics[width=\unitlength,page=1]{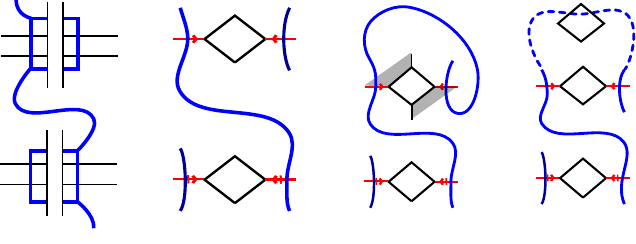}}%
    \put(0.01656482,0.00812618){\color[rgb]{0,0,0}\makebox(0,0)[lt]{\lineheight{1.25}\smash{\begin{tabular}[t]{l}(a)\end{tabular}}}}%
    \put(0.27414081,0.00848967){\color[rgb]{0,0,0}\makebox(0,0)[lt]{\lineheight{1.25}\smash{\begin{tabular}[t]{l}(b)\end{tabular}}}}%
    \put(0.57714469,0.00876554){\color[rgb]{0,0,0}\makebox(0,0)[lt]{\lineheight{1.25}\smash{\begin{tabular}[t]{l}(c)\end{tabular}}}}%
    \put(0.83314624,0.00700964){\color[rgb]{0,0,0}\makebox(0,0)[lt]{\lineheight{1.25}\smash{\begin{tabular}[t]{l}(d)\end{tabular}}}}%
    \put(0.61542209,0.33570309){\color[rgb]{0,0,1}\makebox(0,0)[lt]{\lineheight{1.25}\smash{\begin{tabular}[t]{l}$R$\end{tabular}}}}%
    \put(0,0){\includegraphics[width=\unitlength,page=2]{NoPSSS.pdf}}%
    \put(0.63210258,0.2392196){\color[rgb]{0,0,0}\makebox(0,0)[lt]{\lineheight{1.25}\smash{\begin{tabular}[t]{l}$O$\end{tabular}}}}%
  \end{picture}%
\endgroup%

  \caption{The $SS$ portion of a $PSSS$ or $PPSS$ disk, shown in (a) with saddles in the link complement, and in (b) on the chunk decomposition. In (c) and (d), in fact, it must be a $PSSS$ disk, either with $P$ away from the arc between identified edges as in (c), or with $P$ meeting this arc as in (d).}
  \label{Fig:NoPSSS}
\end{figure}

Each instance of $S$ lies on an interior edge of a chunk that is glued to another interior edge. Hence there must be at least one other curve of  $\bdy Z_i$ running through that glued edge, as shown in \reffig{NoPSSS} (b). But this gives two curves lying on opposite sides of the $PSSS$ or $PPSS$ curve that we consider. Because our disk is innermost or outermost, this is possible only if one of these two curves is actually still part of our original $PPSS$ or $PSSS$ boundary.
If the arc of the $PSSS$ disk between the two identified interior edges cuts off a disk on the projection surface $\Pi$ with those edges, then we have a contradiction to \refprop{NoSaddleTwice}. This was Menasco's argument in~\cite{Menasco1984}. In our setting, it may not be the case that this arc cuts of a disk with the edge, and so \refprop{NoSaddleTwice} may not apply.

So assume that the arc does not cut off a disk on $\Pi$. So far we have found three instances of $S$. It follows that the original disk is a $PSSS$ disk rather than a $PPSS$ disk.

Suppose first that the instance of $P$ does not lie between the two instances of $S$ on identified edges. Hence there is an arc, say $R$, on the boundary of the $PSSS$ disk meeting a single (disk) face of the chunk decomposition, with endpoints on interior edges that are identified to a crossing arc under the gluing. Denote the crossing of $\pi(L)$ that correpsonds to this crossing arc by $O$. By assumption, the disk on $\Pi$ bounded by the $PSSS$ curve does not contain the crossing $O$, or else the arc $R$ would cut off a disk on $\Pi$ that contradicts \refprop{NoSaddleTwice}. Because of checkerboard coloring, the endpoints of the arc $R$ must lie on opposite sides of the crossing $O$, as shown in \reffig{NoPSSS}(c). Then the union of this arc $R$ and another arc running between the endpoints of $R$ through the crossing $O$ forms a closed curve on the surface $\Pi$ that meets the arc $R$ only once (after slight isotopy). This curve is shown by the thin green line in \reffig{NoPSSS}(c). This is impossible: either this curve meets the boundary of the $PSSS$ disk only once on $\Pi$, which is impossible for a disk boundary, or it meets it again as the $PSSS$ disk boundary exits through one of the identified edges at the crossing $O$. But in the latter case, the disk bounded by the $PSSS$ curve in $\Pi$ would then again give a contradiction to \refprop{NoSaddleTwice}.

So the instance of $P$ lies between these two instances of $S$; see \reffig{NoPSSS}(d). This figure shows $\Pi^+$. We now consider how these arcs are glued to $\Pi^-$. Arcs in $\Pi^+$ and in $\Pi^-$ in this case are shown in \reffig{NoPSSS_Finish}. 
\begin{figure}[h]
\begingroup%
  \makeatletter%
  \providecommand\color[2][]{%
    \errmessage{(Inkscape) Color is used for the text in Inkscape, but the package 'color.sty' is not loaded}%
    \renewcommand\color[2][]{}%
  }%
  \providecommand\transparent[1]{%
    \errmessage{(Inkscape) Transparency is used (non-zero) for the text in Inkscape, but the package 'transparent.sty' is not loaded}%
    \renewcommand\transparent[1]{}%
  }%
  \providecommand\rotatebox[2]{#2}%
  \newcommand*\fsize{\dimexpr\f@size pt\relax}%
  \newcommand*\lineheight[1]{\fontsize{\fsize}{#1\fsize}\selectfont}%
  \ifx\svgwidth\undefined%
    \setlength{\unitlength}{167.3998661bp}%
    \ifx\svgscale\undefined%
      \relax%
    \else%
      \setlength{\unitlength}{\unitlength * \real{\svgscale}}%
    \fi%
  \else%
    \setlength{\unitlength}{\svgwidth}%
  \fi%
  \global\let\svgwidth\undefined%
  \global\let\svgscale\undefined%
  \makeatother%
  \begin{picture}(1,0.68843471)%
    \lineheight{1}%
    \setlength\tabcolsep{0pt}%
    \put(0,0){\includegraphics[width=\unitlength,page=1]{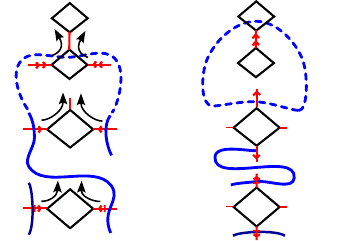}}%
    \put(-0.0031502,0.60850484){\color[rgb]{0,0,0}\makebox(0,0)[lt]{\lineheight{1.25}\smash{\begin{tabular}[t]{l}$\Pi^+$\end{tabular}}}}%
    \put(0.82892678,0.62102168){\color[rgb]{0,0,0}\makebox(0,0)[lt]{\lineheight{1.25}\smash{\begin{tabular}[t]{l}$\Pi^-$\end{tabular}}}}%
  \end{picture}%
\endgroup%

\caption{In case (d) of \reffig{NoPSSS}, the arcs of the $PSSS$ disk in $\Pi^+$ shown on the left must glue to the arcs in $\Pi^-$ shown on the right. }
\label{Fig:NoPSSS_Finish}
\end{figure}
Because the surface $Z$ is in meridianal form, the instance of $P$ on $\Pi^+$ is glued across truncation edges to an instance of $P$ in $\Pi^-$; see \reffig{MeridianalForm}. In particular, the arc in $\Pi^+$ shown in dashed lines in \reffig{NoPSSS}(d), running from $S$ to $P$ and back to an identified $S$, is mapped to a single arc in $\Pi^-$ running from some $S$ to $P$ and back to some $S$. However, on $\Pi^-$, the two endpoints of that arc must be identified together; this is due to the fact that the gluing of chunks rotates the boundary of each face in the clockwise direction in white faces, and the counterclockwise direction in shaded faces. Thus this arc in $\Pi^+$ is identified to a closed curve in $\Pi^-$. It must be the boundary of a disk $Z_i$ by \reflem{4PunctToDisks}. Then this is a $PS$ disk meeting $\Pi^-$. But by \refthm{No2LetterWords}, there are no $PS$ disks, and we have a contradiction. It follows that there are no $PPSS$ or $PSSS$ disks with boundaries in $\Pi^+$.

Now consider $SSSS$ disks. There is no such disk on $\Pi^+$, by \refthm{NoSSSS} and the fact that any compressing disk meets $\Pi^+$ more than four times. Observe that we have now ruled out $SSSS$, $PPSS$, and $PSSS$ disks on $\Pi^+$. By \reflem{4PunctDiskBdry}, there are no further options for disks with two adjacent instances of $S$.

Since the portion of $\bdy Z_i$ running between two instances of $S$ must be glued to another boundary curve running between two instances of $S$, there cannot be a disk with two adjacent instances of $S$ on the other side of $\Pi$ as well. This rules out $SSSS$, $PSSS$, and $PPSS$ disks on $\Pi^-$.

Finally, since there are no $PSSS$ disks, there can be no $PPPS$ disks, because there must be exactly four instances of $P$ on each side of the projection surface, and there are no options for disks meeting only one instance of $P$ to pair with $PPPS$.
\end{proof}

\begin{proof}[Proof of \refthm{EssentialTangles}]
Suppose $Z$ is an essential Conway sphere. By \reflem{4PunctToDisks}, it meets the diagram in normal disks. By \reflem{4PunctDiskBdry} and \reflem{NoPSSSorPPSS}, the only possible disks are labeled $PPPP$ or $PSPS$.

In the case of a $PPPP$ disk, the fact that there must be four instances of $P$ on each side of $\Pi$ implies that there are exactly two $PPPP$ disks, lying on opposite sides of the projection surface. They cut off a visible essential Conway sphere.

In the case of a $PSPS$ disk, there must be four such disks, with one $PSPS$ disk meeting another across each saddle $S$ on the same side of the projection surface. It must additionally meet another two disks at the saddle $S$ on the opposite side of the projection surface. The only possibility is that the four disks fit together with two on one side as in \reffig{VisibleHidden} right. This is a hidden essential Conway sphere.
\end{proof}

\bibliography{biblio}
\bibliographystyle{amsplain}

\end{document}